\pdfoutput=1
\RequirePackage{ifpdf}
\ifpdf % We~are running pdfTeX in pdf mode
\documentclass[pdftex]{sigma}
\else
\documentclass{sigma}
\fi

\numberwithin{equation}{section}

\newtheorem{Theorem}{Theorem}[section]
\newtheorem*{Theorem*}{Theorem}
\newtheorem{Corollary}[Theorem]{Corollary}
\newtheorem{Lemma}[Theorem]{Lemma}
\newtheorem{Proposition}[Theorem]{Proposition}
 { \theoremstyle{definition}
\newtheorem{Definition}[Theorem]{Definition}

 }

\usepackage{enumitem}[2011/09/28]
\setenumerate{align=left, leftmargin=0pt,labelsep=.5em, labelindent=0\parindent,listparindent=\parindent,itemindent=*}

\usepackage{esint,colonequals,mathrsfs,extarrows}

\begin{document}

%\allowdisplaybreaks

\renewcommand{\thefootnote}{}

\newcommand{\arXivNumber}{2401.14099}

\renewcommand{\PaperNumber}{086}

\FirstPageHeading

\ShortArticleName{On a Transformation of Triple $q$-Series and Rogers--Hecke Type Series}

\ArticleName{On a Transformation of Triple $\boldsymbol{q}$-Series\\ and Rogers--Hecke Type Series\footnote{This paper is a~contribution to the Special Issue on Basic Hypergeometric Series Associated with Root Systems and Applications in honor of Stephen C.~Milne's 75th birthday. The~full collection is available at \href{https://www.emis.de/journals/SIGMA/Milne.html}{https://www.emis.de/journals/SIGMA/Milne.html}}}

\Author{Zhi-Guo LIU}

\AuthorNameForHeading{Z.-G.~Liu}

\Address{School of Mathematical Sciences, Key Laboratory of MEA (Ministry of Education) $\&$ Shanghai\\ Key
Laboratory of PMMP, East China Normal University, Shanghai 200241, P.R.~China}
\Email{\href{mailto:zgliu@math.ecnu.edu.cn}{zgliu@math.ecnu.edu.cn}}
\URLaddress{\url{https://math.ecnu.edu.cn/facultydetail.html?uid=zgliu}}

\ArticleDates{Received January 26, 2024, in final form September 15, 2024; Published online October 04, 2024}

\Abstract{Using the method of the $q$-exponential differential operator, we give an extension of the Sears $_4\phi_3$ transformation formula. Based on this extended formula and a $q$-series expansion formula for an analytic function around the origin, we present a transformation formula for triple $q$-series, which includes several interesting special cases, especially a double $q$-series summation formula. Some applications of this transformation formula to Rogers--Hecke type series are discussed. More than 100 Rogers--Hecke type identities including Andrews' identities for the sums of three squares and the sums of three triangular numbers are obtained.}


\Classification{05A30; 33D05; 33D15; 32A05; 11E25; 32A10}

\begin{flushright}
\emph{Dedicated to Stephen Milne on his 75th birthday}
\end{flushright}

\renewcommand{\thefootnote}{\arabic{footnote}}
\setcounter{footnote}{0}

\section{Introduction and preliminaries}

In this paper, we will prove a transformation formula for triple $q$-hypergeometric series, which includes several interesting special cases. One special case of this transformation allows us to derive several new identities for Rogers--Hecke type series. Rogers--Hecke type series play an important role in the study of Ramanujan's mock-theta functions.

As usual, we use $\mathbb{C}$ to denote the set of all complex numbers. Throughout the paper, we shall use the standard $q$-notations. Unless stated otherwise, it is assumed that $0<|q|<1$.
Let $n$ be an arbitrary nonnegative integer and $a \in \mathbb{C}$.
We define the $q$-shifted factorial by
\begin{equation*}
	(a;q)_n=\frac{(a; q)_\infty}{(aq^n; q)_\infty}, \qquad \text{where}\quad
	(a; q)_\infty=\prod_{k=0}^\infty \bigl(1-aq^k\bigr).
\end{equation*}
It is obvious that $(a; q)_0=1$, and by a simple calculation we find that for any positive integer~$n$,
\begin{gather*}
(a; q)_n=(1-a)(1-aq)\cdots \bigl(1-aq^{n-1}\bigr),\\
(a; q)_{-n}=\frac{1}{\bigl(aq^{-1}; q^{-1}\bigr)_n}=(-a)^{-n} q^{n(n+1)/2} (q/a; q)_n^{-1}.
\end{gather*}
For $n$ being an integer or infinity, the multiple
$q$-shifted factorial for $a_1, a_2, \ldots, a_m $ is defined by
\begin{equation*}
(a_1, a_2,\dots ,a_m;q)_n=(a_1;q)_n(a_2;q)_n \cdots (a_m;q)_n.
\end{equation*}

The basic hypergeometric series or the $q$-hypergeometric series ${_{r}\phi_s}(\cdot)$ are defined as
\begin{equation*}
_{r}\phi_s \biggl({{a_1, \ldots, a_{r}} \atop {b_1, \ldots,
b_s}} ; q, z \biggr) =\sum_{n=0}^\infty
\frac{(a_1, \ldots, a_r; q)_n}{(q, b_1, \ldots, b_s; q)_n}
\bigl((-1)^n q^{n(n-1)/2}\bigr)^{1+s-r}z^n.
\end{equation*}

One of the most important results in the theory of $q$-hypergeometric series is the $q$-binomial theorem, which is stated in the following theorem \cite[p.~488]{AAR1999}.

\begin{Theorem}[$q$-binomial theorem] %\label{binomialthm}
For $|q|<1$ and $|bx|<1$, we have
	\begin{equation*}
		\frac{(ax; q)_\infty}{(bx; q)_\infty}
		=\sum_{n=0}^\infty \frac{b^n (a/b; q)_n}{(q; q)_n} x^n.
	\end{equation*}
\end{Theorem}
Setting $a=0$ and $b=0$ in the above equation respectively, we arrive at the following two formulas due to Euler:
\begin{align*}%\label{qbinom:eqn1}
	&\frac{1}{(bx; q)_\infty}=\sum_{n=0}^\infty \frac{(bx)^n}{(q; q)_n},\qquad |bx|<1,
\end{align*}
and
\begin{align*}
	&(ax; q)_\infty=\sum_{n=0}^\infty \frac{(-ax)^n q^{n(n-1)/2}}{(q; q)_n}, \qquad |ax|<\infty. %\label{qbinom:eqn2}
\end{align*}

By a simple calculation, it is found that the $q$-Pfaff--Saalsch\"{u}tz formula \cite[equation~(1.7.2)]{Gas+Rah} can be rewritten as
\begin{equation}\label{qpfaff}
{_3\phi_2}\biggl({{q^{-n}, \alpha q^n, \alpha ab/q}\atop{\alpha a, \alpha b}}; q, q\biggr)=\frac{(q/a, q/b; q)_n}{(\alpha a, \alpha b; q)_n}
\biggl(\frac{\alpha ab}{q}\biggr)^n.
\end{equation}
Setting $b=0$ in the equation above, we arrive at the $q$-Chu--Vandermonde summation \cite[equa\-tion~(1.5.3)]{Gas+Rah}
\begin{equation}\label{qchu-van}
	{_2\phi_1}\biggl({{q^{-n}, \alpha q^n}\atop{\alpha a}}; q, q\biggr)=\frac{(q/a; q)_n (-\alpha a)^n}{(\alpha a; q)_n} q^{n(n-1)/2}.	
\end{equation}

The Rogers $_6\phi_5$ summation or equivalently, the non-terminating $_6\phi_5$ summation \cite[p.~44]{Gas+Rah}~is also one of the central theorems in $q$-series, which is stated in the following proposition.
\begin{Proposition} \label{rogersthm} For $\bigl|\alpha abc/q^2\bigr|<1,$ we have
	\begin{align*}
		{_6 \phi_5} \biggl({{\alpha, q\sqrt{\alpha}, -q\sqrt{\alpha}, q/a, q/b, q/c}
			\atop{\sqrt{\alpha}, -\sqrt{\alpha},\alpha a, \alpha b, \alpha c}}; q, \frac{\alpha abc}{q^2}\biggr)
		=\frac{(\alpha q, \alpha ab/q, \alpha ac/q, \alpha bc/q; q)_\infty}
		{\bigl(\alpha a, \alpha b, \alpha c, \alpha abc/q^2; q\bigr)_\infty}.
	\end{align*}
\end{Proposition}

Using the method of the $q$-exponential differential operator,
the following extension of the Rogers $_6\phi_5$ summation formula is given by us in \cite[Theorem~3]{Liu2011}, see also \cite[Theorem~6.2]{LiuRam2013}.
\begin{Theorem}\label{rogersliuthm}
	For $\max\bigl\{\bigl|\alpha \beta abc/q^2\bigr|, \bigl|\alpha \gamma abc/q^2\bigr|\bigr\}<1$, we have
	\begin{align*}
		&\sum_{n=0}^\infty \frac{\bigl(1-\alpha q^{2n}\bigr)(\alpha, q/a, q/b, q/c; q)_n}{(1-\alpha)(q, \alpha a, \alpha b, \alpha c; q)_n}
		\biggl(\frac{\alpha abc}{q^2}\biggr)^n {_4 \phi_3} \biggl({{q^{-n}, \alpha q^n, \beta, \gamma}
			\atop{q/a, q/b,\alpha \beta \gamma ab/q}}; q, q\biggr)\\
		&\qquad{}=\frac{\bigl(q\alpha, \alpha ac/q, \alpha bc/q, \alpha \beta ab/q, \alpha \gamma ab/q, \alpha \beta \gamma abc/q^2; q\bigr)_\infty}
		{\bigl(\alpha a, \alpha b, \alpha c, \alpha \beta abc/q^2, \alpha \gamma abc/q^2, \alpha \beta \gamma ab/q; q\bigr)_\infty}.
	\end{align*}
\end{Theorem}
 This summation formula allows us to give a beautiful generating function for the Askey--Wilson polynomials, which leads us
 to derive a twelve-parameter $q$-beta integral in \cite{Liu2019} and give a new proof of the orthogonality of the Askey--Wilson polynomials in \cite{Liu2021}.

It is generally known that it is a difficult task to obtain triple $q$-series transformation formulas, and interesting transformation formulas of triple $q$-series are rare in the mathematical literature.

In this paper, we will further extend the double $q$-summation formula in Theorem~\ref{rogersliuthm} to
the following transformation formula for triple $q$-hypergeometric series with the help of an extension of the Sears $_4\phi_3$ transformation.
\begin{Theorem}\label{liuthma} We have
	\begin{align*}
		&\sum_{n=0}^\infty \frac{\bigl(1-\alpha q^{2n}\bigr)(\alpha, q/a, q/b, q/c; q)_n}{(1-\alpha)(q, \alpha a, \alpha b, \alpha c; q)_n}
		\biggl(\frac{\alpha abc}{q^2}\biggr)^n \\
		&\qquad\quad{} \times \sum_{k=0}^{n} \frac{(q^{-n}, \alpha q^n, \beta, \gamma; q)_k q^k}
		{(q, q/a, q/b, \beta \gamma uv; q)_k}
		{_3\phi_2 \biggl({{q^{-k}, 1/u, 1/v}\atop{\beta, \gamma}}; q, \beta \gamma uv q^k\biggr)}\\
		&\qquad{}=\frac{(q\alpha, \alpha ac/q, \alpha bc/q, \alpha \beta ab/q, \alpha \gamma ab/q, \beta \gamma cuv/q; q )_\infty}{\bigl(\alpha a, \alpha b, \alpha c, \alpha \beta abc/q^2, \alpha \gamma abc/q^2, \beta \gamma uv; q \bigr)_\infty}\\
		&\qquad\quad{} \times{_3\phi_2} \biggl({{q/c, \alpha ab/qu, \alpha ab/qv}\atop{\alpha \beta ab/q, \alpha \gamma ab/q}}; q,
		\beta \gamma cuv/q \biggr),
	\end{align*}
	provided the series on both sides of the above equation converge absolutely.
\end{Theorem}
 The Kronecker delta function $\delta_{mn}$ is defined as $\delta_{mn}=1$ when $m=n$ and $\delta_{mn}=0$ when~${m\not=n}$.

Setting $qv=\alpha ab$ in Theorem~\ref{liuthma}, upon noting that $(1; q)_k=\delta_{0k},$ the $_3\phi_2$ series on the right-hand side in the theorem reduces to $1$, and thus we obtain the following corollary.
\begin{Corollary}\label{liuthmb} We have
	\begin{align*}
		&\sum_{n=0}^\infty \frac{\bigl(1-\alpha q^{2n}\bigr)(\alpha, q/a, q/b, q/c; q)_n}{(1-\alpha)(q, \alpha a, \alpha b, \alpha c; q)_n}
		\biggl(\frac{\alpha abc}{q^2}\biggr)^n \\
		&\qquad\quad{} \times \sum_{k=0}^{n} \frac{(q^{-n}, \alpha q^n, \beta, \gamma; q)_k q^k}
		{(q, q/a, q/b, \alpha \beta \gamma ab u/q; q)_k}
		{_3\phi_2 \biggl({{q^{-k}, 1/u, q/\alpha ab}\atop{\beta, \gamma}}; q, \alpha \beta \gamma ab u q^{k-1}\biggr)}\\
		&\qquad{}=\frac{\bigl(q\alpha, \alpha ac/q, \alpha bc/q, \alpha \beta ab/q, \alpha \gamma ab/q, \alpha \beta \gamma abcu/q^2; q \bigr)_\infty}{\bigl(\alpha a, \alpha b, \alpha c, \alpha \beta abc/q^2, \alpha \gamma abc/q^2, \alpha \beta \gamma abu/q; q\bigr)_\infty}.
\end{align*}
\end{Corollary}
When $u=1, $ the ${_3}\phi{_2}$ series in the above equation reduces to $1$ and Corollary~\ref{liuthmb} becomes
Theorem~\ref{rogersliuthm}. Hence, Theorem~\ref{liuthma} is really an extension of Theorem~\ref{rogersliuthm}.

Setting $v=1$ in Theorem~\ref{liuthma}, upon noting that $(1; q)_k=\delta_{0k},$ the $_3\phi_2$ series
on the left-hand side in the theorem becomes $1$, we find the following corollary.

\begin{Corollary}\label{liuthmc} We have
	\begin{align*}
		&\sum_{n=0}^\infty \frac{\bigl(1-\alpha q^{2n}\bigr)(\alpha, q/a, q/b, q/c; q)_n}{(1-\alpha)(q, \alpha a, \alpha b, \alpha c; q)_n}\biggl(\frac{\alpha abc}{q^2}\biggr)^n {_4\phi_3}\biggl({{q^{-n}, \alpha q^n, \beta, \gamma}
			\atop{q/a, q/b, \beta \gamma u}}; q, q\biggr)\\
		&\qquad{}=\frac{(q\alpha, \alpha ac/q, \alpha bc/q, \alpha \beta ab/q, \alpha \gamma ab/q, \beta \gamma cu/q; q )_\infty}{\bigl(\alpha a, \alpha b, \alpha c, \alpha \beta abc/q^2, \alpha \gamma abc/q^2, \beta \gamma u; q \bigr)_\infty}\\
		&\qquad\quad{} \times{_3\phi_2} \biggl({{q/c, \alpha ab/qu, \alpha ab/q}\atop{\alpha \beta ab/q, \alpha \gamma ab/q}}; q, \beta \gamma cu/q \biggr).
	\end{align*}
\end{Corollary}
Setting $u=\alpha ab/q $ in the above corollary, upon noting that $(1; q)_k=\delta_{0k},$ the $_3\phi_2$ series
on the right-hand side in the theorem becomes $1$, we arrive at Theorem~\ref{rogersliuthm}.

The main tool to prove Theorem~\ref{liuthma} is the following
$q$-expansion formula, which is equivalent to \cite[Theorem~2]{Liu2002} and \cite[Theorem~1.1]{Liu2013RamanJ}.
\begin{Theorem}\label{liuthm} If $f(x)$ is an analytic function near $x=0$, then we have
	\begin{align}\label{newliuexpthm}
		f(\alpha a)
		=\sum_{n=0}^\infty \frac{\bigl(1-\alpha q^{2n}\bigr)(\alpha, q/a; q)_n (a/q)^n}
		{(1-\alpha)(q, \alpha a; q)_n} \sum_{k=0}^n \frac{(q^{-n}, \alpha q^n; q)_k q^k}
		{(q, q\alpha; q)_k} f\bigl(\alpha q^{k+1}\bigr).
	\end{align}
\end{Theorem}
The rest of the paper is organized as follows. In Section~\ref{sec2}, we will provide a proof of Theorem~\ref{liuthm}. In Section~\ref{sec3}, we will give an extension of the Sears $_4\phi_3$ transformation formula and provide a proof of Theorem~\ref{liuthma}. In Section~\ref{sec4}, we will investigate the application of Theorem~\ref{rogersliuthm} in number theory, and many results including Andrews' identities for the sums of three squares and the sums of three triangular numbers are proved.

 The following new transformation formula for $q$-series is proved in Section~\ref{sec5}, it is very useful in deriving identities for Rogers--Hecke type series.
\begin{Theorem}\label{trianHecke} The following double $q$-transformation formula holds:
\begin{align}
	&\sum_{n=0}^\infty \sum_{j=0}^{2n} \bigl(1-q^{4n+2}\bigr) q^{2n^2-j(j+1)/2}
	\frac{\bigl(q^2/a, q^2/c; q^2\bigr)_n}{\bigl(q^2a, q^2c; q^2\bigr)_n} (ac)^n\nonumber\\
	&\qquad{}=\frac{\bigl(ac, q^2; q^2\bigr)_\infty}{\bigl(q^2a, q^2c; q^2\bigr)_\infty}
	\sum_{n=0}^\infty \frac{\bigl(q^2/a, q^2/c; q^2\bigr)_n}{(q; q)_{2n}}
	\biggl(\frac{ac}{q}\biggr)^n.\label{trinumbers: eqn1}
\end{align}		
\end{Theorem}	
When $a=c=0$, Theorem~\ref{trianHecke} immediately reduces to the following identity of Rogers--Hecke type, which seems to be new,
\begin{equation}\label{trinumbers:eqn2}
\sum_{n=0}^\infty \frac{q^{2n^2+n}}{(q; q)_{2n}}
=\frac{1}{\bigl(q^2; q^2\bigr)_\infty}
\sum_{n=0}^\infty \sum_{j=0}^{2n} \bigl(1-q^{4n+2}\bigr) q^{4n^2+2n-j(j+1)/2}.
\end{equation}
Choosing $(a, c)=(1, 0)$ in Theorem~\ref{trianHecke}, we deduce that
\begin{equation}\label{trinumbers:eqn3}
\sum_{n=0}^\infty (-1)^n \frac{q^{n^2}}{\bigl(q; q^2\bigr)_n}
=\sum_{n=0}^\infty \sum_{j=0}^{2n} (-1)^n \bigl(1-q^{4n+2}\bigr) q^{3n^2+n-j(j+1)/2}.
\end{equation}
Replacing $q$ by $-q$ in the equation above and simplifying, we conclude that
\begin{equation*}%\label{trinumbers: eqn33}
\sum_{n=0}^\infty \frac{q^{n^2}}{\bigl(-q; q^2\bigr)_n}
=\sum_{n=0}^\infty \sum_{j=0}^{2n} (-1)^{n+j(j+1)/2} \bigl(1-q^{4n+2}\bigr) q^{3n^2+n-j(j+1)/2}.
\end{equation*}
By making the change $(a, c)\to (-1, 0)$ in Theorem~\ref{trianHecke}, we obtain that
\begin{equation*}%\label{trinumbers:eqn4}
	\sum_{n=0}^\infty \frac{\bigl(-q^2; q^2\bigr)_n q^{n^2}}{(q; q)_{2n}}
	=\frac{\bigl(-q^2; q^2\bigr)_\infty}{\bigl(q^2; q^2\bigr)_\infty} \sum_{n=0}^\infty \sum_{j=0}^{2n} \bigl(1-q^{4n+2}\bigr) q^{3n^2+n-j(j+1)/2}.	
\end{equation*}
Letting $(a, c)=(q, 0)$ in Theorem~\ref{trianHecke} and performing a simple calculation, we deduce that
\begin{equation*}%\label{ndd:eqn1}
\sum_{n=0}^{\infty} (-1)^n \frac{q^{n^2+n}}{\bigl(q^2; q^2\bigr)_n}
=\frac{\bigl(q; q^2\bigr)_\infty}{\bigl(q^2; q^2\bigr)_\infty}
\sum_{n=0}^{\infty}\sum_{j=0}^{2n} (-1)^n \bigl(1+q^{2n+1}\bigr) q^{3n^2+2n-j(j+1)/2}.
\end{equation*}

If we specialize Theorem~\ref{trianHecke} to the case when $a=c=q$, we conclude that
\begin{equation*}%\label{trinumbers: eqn5}
\sum_{n=0}^\infty \frac{\bigl(q; q^2\bigr)_n q^n}{\bigl(q^2; q^2\bigr)_n}
=\frac{\bigl(q; q^2\bigr)^2_\infty}{\bigl(q^2; q^2\bigr)^2_\infty}
\sum_{n=0}^\infty \sum_{j=0}^{2n} \biggl(\frac{1+q^{2n+1}}{1-q^{2n+1}}\biggr)
q^{2n^2+2n-j(j+1)/2}.
\end{equation*}

It is obvious that when $c=q$, the right-hand side of (\ref{trinumbers: eqn1}) can be summed by the $q$-binomial theorem, namely,
\begin{equation*}%\label{trinumbers: eqn6}
	\sum_{n=0}^\infty \frac{\bigl(q^2/a; q^2\bigr)_n a^n}{\bigl(q^2; q^2\bigr)_n}
	=\frac{\bigl(q^2; q^2\bigr)_\infty}{\bigl(a; q^2\bigr)_\infty}.
\end{equation*}

Hence, setting $c=q$ in (\ref{trinumbers: eqn1}) and then using the equation above, we arrive at the following remarkable $q$-formula:
\begin{equation}\label{trinumbers: eqn7}
\frac{\bigl(q^2, q^2, qa; q^2\bigr)_\infty}{\bigl(q, a, q^2a; q^2\bigr)_\infty}
=\sum_{n=0}^\infty \sum_{j=0}^{2n} \bigl(1+q^{2n+1}\bigr) q^{2n^2+n-j(j+1)/2}
\frac{\bigl(q^2/a; q^2\bigr)_n a^n}{\bigl(q^2a; q^2\bigr)_n}.
\end{equation}

Setting $a=q$ in the equation above and then using the Gauss identity
(see, for example~\mbox{\cite[p.~347]{Liu2010pacific})}
\[
\frac{\bigl(q^2; q^2\bigr)_\infty}{\bigl(q; q^2\bigr)_\infty}=\sum_{n=0}^\infty q^{n(n+1)/2}
\]
in the resulting equation, we arrive at Andrews' identity for the sums of three triangular num\-bers~\cite[equation~(1.5)]{Andrews86}:
\begin{equation}\label{trinumbers:eqn8}
\Biggl(\sum_{n=0}^\infty q^{n(n+1)/2}\Biggr)^3
=\sum_{n=0}^\infty \sum_{j=0}^{2n} \biggl(\frac{1+q^{2n+1}}{1-q^{2n+1}}\biggr) q^{2n^2+2n-j(j+1)/2},
\end{equation}
which implies Gauss' famous result that every integer is the sum of three triangular numbers.

Putting $a=-q$ and $a=0$ respectively in (\ref{trinumbers: eqn7}), we deduce the following two identities:
\begin{equation*}%\label{trinumbers:eqn9}
\frac{\bigl(q^2, q^2, -q^2; q^2\bigr)_\infty}{\bigl(q, -q, -q; q^2\bigr)_\infty}
=\sum_{n=0}^\infty \sum_{j=0}^{2n} (-1)^n q^{2n^2+2n-j(j+1)/2}
\end{equation*}
and
\begin{equation*}%\label{trinumbers:eqn10}
\frac{\bigl(q^2; q^2\bigr)^2_\infty}{\bigl(q; q^2\bigr)_\infty}
=\sum_{n=0}^\infty \sum_{j=0}^{2n} (-1)^n \bigl(1+q^{2n+1}\bigr) q^{3n^2+2n-j(j+1)/2}.
\end{equation*}

Some applications of Corollary~\ref{liuthmc} to Hecke-type series are discussed in Section~\ref{sec6}. In particular, for any nonnegative integer $m$, we prove that
\begin{gather}\label{finitemock}
\sum_{n=0}^\infty \frac{q^{n(n+1)/2}}{(-q; q)_n \bigl(1+q^{m+n+1}\bigr)}
=\sum_{n=0}^m \sum_{j=-n}^n (-1)^{n+j} \frac{\bigl(1-q^{2n+1}\bigr)(q; q)^2_m}
{(q; q)_{m-n} (q; q)_{m+n+1}} q^{(3n^2+n)/2-j^2}.
\end{gather}

In Section~\ref{sec7}, we will prove the following theorem among others.
\begin{Theorem}\label{liuheckethm} For $a\not=q^{-m}$, $m=0, 1, 2, \ldots$, we have
	\begin{align}\label{liuhr:eqn1}
		&\frac{(-a, q; q)_\infty} {(a, -q; q)_\infty}\sum_{n=0}^\infty \frac{q^{n(n-1)/2}}{(-a; q)_n}\\
		&\qquad{}=\sum_{n=0}^\infty \sum_{j=-n}^n (-1)^{n+j} q^{(3n^2-n)/2-j^2} \frac{(q/a; q)_n a^n}{(a; q)_n}
		\biggl(1+q^n+\bigl(1+q^{n+1}\bigr)\frac{q^{2n}\bigl(a-q^{n+1}\bigr)}{1-aq^n}\biggr)\nonumber
	\end{align}		
	or
	\begin{align}
\frac{(-a, q; q)_\infty} {(a, -q; q)_\infty}\sum_{n=0}^\infty \frac{q^{n(n-1)/2}}{(-a; q)_n}
		 &{}=2+2\sum_{n=1}^\infty (1+q^n) \frac{(q/a; q)_n a^n}{(a; q)_n} q^{n(n-1)/2}\nonumber\\
		& \quad{}-\sum_{n=1}^\infty \bigl(1-q^{2n}\bigr)\frac{(q/a; q)_n (-a)^n}{(a; q)_n} \sum_{|j|<n} (-1)^j
		q^{3n(n-1)/2-j^2}.\label{liuhr:eqn2}
	\end{align}
\end{Theorem}	
 Setting $a=q$ in (\ref{liuhr:eqn2}), we immediately get the following remarkable identity of Andrews \mbox{\cite[equation~(4.4)]{Andrews1967}} (see also \cite[equation~(3.62)]{ChenWang2020}):
 \begin{equation*}%\label{liuhr:eqn3}
 \sum_{n=0}^\infty \frac{q^{n(n-1)/2}}{(-q; q)_n}=2.
 \end{equation*}
Replacing $q$ by $q^2$ in (\ref{liuhr:eqn1}) and then taking $a=q$ and $a=-q$, respectively, in the resulting equation, we deduce that
\begin{equation*}%\label{newadd:eqn1}
\sum_{n=0}^\infty \frac{q^{n(n-1)}}{\bigl(-q; q^2\bigr)_n}
=\frac{\bigl(q, -q^2; q^2\bigr)_\infty}{\bigl(-q, q^2; q^2\bigr)_\infty}
\sum_{n=0}^\infty \sum_{j=-n}^n (-1)^{n+j} \bigl(1+q^{2n}+q^{4n+1}+q^{6n+3}\bigr)
q^{3n^2-2j^2}
\end{equation*}
and
\begin{equation*}%\label{newadd:eqn2}
	\sum_{n=0}^\infty \frac{q^{n(n-1)}}{\bigl(q; q^2\bigr)_n}
	=\frac{(-q; q)_\infty}{(q; q)_\infty}
	\sum_{n=0}^\infty \sum_{j=-n}^n (-1)^{j} \bigl(1+q^{2n}-q^{4n+1}-q^{6n+3}\bigr)
	q^{3n^2-2j^2}.
\end{equation*}

In Section~\ref{sec8}, we will derive more identities of Rogers--Hecke type.

\section[A q-series expansion formula and evaluation of some terminating q-series]{A $\boldsymbol{q}$-series expansion formula and evaluation \\ of some terminating $\boldsymbol{q}$-series}\label{sec2}

\subsection[A q-series expansion formula]{A $\boldsymbol{q}$-series expansion formula}
In order to prove Theorem~\ref{liuthm}, we need to review some basic knowledge of $q$-calculus.

The $q$-derivative was introduced by L. Schendel \cite{Schendel} in 1878 and F.H.~Jackson \cite{Jackson1908} in 1908, which is a $q$-analog of the ordinary derivative.
\begin{Definition}%\label{qddfn}
	If $q$ is a complex number that is neither $0$ nor $1$, then for any function $f(x)$ of one variable, the $q$-derivative of $f(x)$
	with respect to $x$ is defined as
	\begin{equation*}
		{D}_{q}f(x)=\frac{f(x)-f(qx)}{x},
	\end{equation*}
	and we further define ${D}_{q}^0 f=f,$ and for $n\ge 1$, ${D}_{q}^n f={D}_{q}\bigl\{{D}_{q}^{n-1}f\bigr\}.$
\end{Definition}

Using mathematical induction, one can easily derive the following formula of Jackson \cite{Jackson1908}, which writes
${D}^n_{q}f(x) $ in terms of $f\bigl(q^k x\bigr)$ for $k=0, 1, \ldots, n$.
\begin{Lemma}\label{JacsonLem} For any nonnegative integer $n,$ we have
	\[
	{D}^n_{q}f(x)=x^{-n} \sum_{k=0}^n \frac{(q^{-n}; q)_k}{(q; q)_k}q^k f\bigl(q^k x\bigr).
	\]
\end{Lemma}

 \begin{proof} [Proof of Theorem~\ref{liuthm}] If $f(x)$ is a formal power series in $x$ over the field of complex numbers~$\mathbb{C}$, then we have the following expansion formula \cite[Theorem~2]{Liu2002}:
	\begin{equation}\label{expliu:eqn1}
		f(\alpha a)=\sum_{n=0}^\infty \frac{\bigl(1-\alpha q^{2n}\bigr)(q/a; q)_n (\alpha a)^n}{(q, \alpha a;
			q)_n}\bigl[ {D}^n_{q}\{f(x)(x; q)_{n-1}\} \bigr]_{x=\alpha q}.
\end{equation}		
Using the Jackson formula in Lemma~\ref{JacsonLem} and the definition of $q$-shifted factorial, we have the following evaluation:
	\begin{align*}
		{D}^n_{q}\{f(x)(x; q)_{n-1}\} &{}=x^{-n}\sum_{k=0}^n \frac{(q^{-n}; q)_k}{(q; q)_k}q^k f\bigl(q^k x\bigr)\bigl(xq^k; q\bigr)_{n-1}\\
		&{}=x^{-n} (x; q)_{n-1}\sum_{k=0}^n \frac{\bigl(q^{-n}, xq^{n-1}; q\bigr)_k}{(q, x; q)_k}q^k f\bigl(q^k x\bigr),
	\end{align*}	
	which gives
	\[
	\bigl[ {D}^n_{q}\{f(x)(x; q)_{n-1}\} \bigr]_{x=\alpha q}
	=\frac{({\alpha}; q)_n }{(1-\alpha)(q\alpha)^{n}}
	\sum_{k=0}^n \frac{(q^{-n}, \alpha q^n; q)_k q^k}
	{(q, q\alpha; q)_k} f\bigl(\alpha q^{k+1}\bigr).
	\]
	Combining this equation with (\ref{expliu:eqn1}) completes the proof of Theorem~\ref{liuthm}.
 \end{proof}

Theorem~\ref{liuthm} is a convenient tool for deriving $q$-identities, and by appropriately choosing $f(x)$ in Theorem~\ref{liuthm} we can derive many expansion formulas for $q$-series. Now we will provide a proof of Proposition~\ref{rogersthm} using Theorem~\ref{liuthm}.
\begin{proof}[Proof of Proposition~\ref{rogersthm}]
In Theorem~\ref{liuthm}, we choose $f(x)$ as follows:
\[
f(x)=\frac{(bx/q, cx/q; q)_\infty}{\bigl(x, bcx/q^2; q\bigr)_{\infty}}.
\]
A straightforward computation shows that
\begin{align*}
f(\alpha a)=\frac{(\alpha ab/q, \alpha ac/q; q)_\infty}{\bigl(\alpha a, \alpha a bc/q^2; q\bigr)_\infty}
\end{align*}
and
\begin{align*}
f\bigl(\alpha q^{k+1}\bigr)=\frac{(\alpha b, \alpha c; q)_\infty (q\alpha, \alpha bc/q; )_k}{(q\alpha, \alpha bc/q)_\infty (\alpha b, \alpha c; q)_k}.
\end{align*}
Substituting these two equations into the $q$-expansion formula (\ref{newliuexpthm}), we conclude that
{\samepage\begin{align*}	
&\frac{(q\alpha, \alpha ab/q, \alpha ac/q, \alpha bc/q; q)_\infty}
{\bigl(\alpha a, \alpha b, \alpha c, \alpha abc/q^2; q\bigr)_\infty}\\
&\qquad{}=\sum_{n=0}^\infty \frac{\bigl(1-\alpha q^{2n}\bigr)(\alpha, q/a; q)_n (a/q)^n}{(1-\alpha)(\alpha a; q)_n}	
\sum_{k=0}^n \frac{(q^{-n}, \alpha q^n, \alpha bc/q; q)_n q^k}
{(q, \alpha b, \alpha c; q)_n}.
\end{align*}}	
 Using the $q$-Pfaff--Saalsch\"utz summation formula (\ref{qpfaff}), it is found that
\begin{equation*}%\label{expliu:eqn2}
\sum_{k=0}^n \frac{(q^{-n}, \alpha q^n, \alpha bc/q; q)_n q^k}
{(q, \alpha b, \alpha c; q)_n}
=\frac{(q/b, q/c; q)_n}{(\alpha b, \alpha c; q)_n}\biggl(\frac{\alpha bc}{q}\biggr)^n.
\end{equation*}
Combining the above two equations, we immediately arrive at the Rogers $_6\phi_5$ summation in Proposition~\ref{rogersthm}.
\end{proof}

In order to indicate the power of Theorem~\ref{liuthm} again, we will use it to derive Watson's $q$-analogue of Whipple's theorem.

\subsection[Watson's q-analogue of Whipple's theorem]{Watson's $\boldsymbol{q}$-analogue of Whipple's theorem}

Watson's $q$-analog of Whipple's theorem (see, for example, \cite[equation~(2.5.1)]{Gas+Rah}) is stated in the following theorem.
\begin{Theorem}\label{WatsonWhipplethm} For any nonnegative integer $n$, we have
	\begin{align*}%\label{expliu:eqn3}
		&\frac{(\alpha q, \alpha ab/q; q)_n}{(\alpha a, \alpha b; q)_n}
		{_4\phi_3}\biggl({{q^{-n}, q/a, q/b, \alpha cd/q }\atop{\alpha c, \alpha d, q^2/{\alpha ab q^n}}}; q, q\biggr)\\
		&\qquad{}={_8\phi_7}\biggl({{q^{-n}, q\sqrt{\alpha}, -q\sqrt{\alpha}, \alpha, q/a, q/b, q/c, q/d}
			\atop{\sqrt{\alpha}, -\sqrt{\alpha}, \alpha a, \alpha b, \alpha c, \alpha d, \alpha q^{n+1}}}; q, \frac{\alpha^2 abcdq^n}{q^2}\biggr).\nonumber
	\end{align*}
\end{Theorem}
\begin{proof} For any nonnegative integer $m$, if we specialize Theorem~\ref{liuthm} to the case when $f(x)=(bx/q; q)_m/{(x; q)_m}$, we immediately have
\begin{equation*}%\label{expliu:eqn4}
f(\alpha a)=\frac{(\alpha ab/q; q)_m}{(\alpha a; q)_m}
\end{equation*}	
and
\begin{equation*}%\label{expliu:eqn5}
f\bigl(\alpha q^{k+1}\bigr)=\frac{\bigl(\alpha bq^{k}; q \bigr)_m}{\bigl(\alpha q^{k+1}; q\bigr)_m}=\frac{(\alpha q; q)_k (\alpha b;q)_{k+m}}{(\alpha b; q)_k (q\alpha ; q)_{k+m}}.	
\end{equation*}	
Now begin to evaluate the inner sum on the right-hand side of the equation in Theorem~\ref{liuthm}. It is easily seen that
\begin{equation*}%\label{expliu:eqn6}
\sum_{k=0}^n \frac{(q^{-n}, \alpha q^n; q)_k q^k}
{(q, q\alpha; q)_k} f\bigl(\alpha q^{k+1}\bigr)
=\frac{(\alpha b; q)_m}{(q\alpha; q)_m}
\sum_{k=0}^n \frac{(q^{-n}, \alpha q^n, \alpha b q^m; q)_k q^k}
{\bigl(q, \alpha b, q^{m+1}\alpha; q\bigr)_k}.
\end{equation*}
Applying the $q$-Pfaff--Saalsch\"{u}tz summation formula (\ref{qpfaff}) to the summation on the right-hand side of the above equation, we find that
\begin{align*}%\label{expliu:eqn7}
\sum_{k=0}^n \frac{(q^{-n}, \alpha q^n; q)_k q^k}
{(q, q\alpha; q)_k} f\bigl(\alpha q^{k+1}\bigr)&{}=\frac{(\alpha b; q)_m (q/b, q^{-m}; q)_n}{(q\alpha; q)_m
(\alpha b, q^{m+1}\alpha; q)_n} (\alpha b q^m)^n\\
&{}=(-1)^n \frac{(\alpha b; q)_m (q/b; q)_n (q; q)_m(\alpha b)^n}{(\alpha b; q)_n (q\alpha; q)_{m+n} (q; q)_{m-n}} q^{n(n-1)/2}.\nonumber
\end{align*}
Substituting the equation above into Theorem~\ref{liuthm}, we conclude that
\begin{align*}%\label{expliu:eqn8}
&\frac{(\alpha ab/q; q)_m}{(q, \alpha a, \alpha b; q)_m}
=\sum_{n=0}^m \frac{\bigl(1-\alpha q^{2n}\bigr) (\alpha, q/a, q/b; q)_n}{(1-\alpha)(q, \alpha a, \alpha b; q)_n (q; q)_{m-n} (q\alpha; q)_{m+n}}\biggl(-\frac{\alpha ab}{q}\biggr)^n q^{n(n-1)/2}. \nonumber
\end{align*}
Let $N$ be a nonnegative integer. Multiplying both sides of the above equation by
\[
\frac{\bigl(q^{-N}, q/c, q/d; q\bigr)_m q^m}{\bigl(q^2/\alpha cdq^N; q\bigr)_m}
\]
and then summing the resulting equation from $m=0$ to $m=N$, and interchanging the order of summation, we find that
\begin{align}
{_4\phi_3}\biggl({{q^{-N}, q/c, q/d, \alpha ab/q }\atop{\alpha a, \alpha b, q^2/{\alpha cd q^N}}}; q, q\biggr)
&{}=\sum_{n=0}^N \frac{\bigl(1-\alpha q^{2n}\bigr) (\alpha, q/a, q/b; q)_n}{(1-\alpha)(q, \alpha a, \alpha b; q)_n }\biggl(-\frac{\alpha ab}{q}\biggr)^n q^{n(n-1)/2}\nonumber\\
&{}\quad \times \sum_{m=n}^N \frac{\bigl(q^{-N}, q/c, q/d ; q\bigr)_m q^m}
{(q; q)_{m-n} \bigl(q^2/{\alpha cd q^N; q}\bigr)_{m} (q\alpha; q)_{m+n}}.\label{expliu:eqn9}
\end{align}
Making the variable change $m=n+j$ and using the $q$-Pfaff--Saalsch\"{u}tz summation again, we deduce that
\begin{align*}%\label{expliu:eqn10}
&\sum_{m=n}^N \frac{\bigl(q^{-N}, q/c, q/d ; q\bigr)_m q^m}
{(q; q)_{m-n} \bigl(q^2/{\alpha cd q^N}; q\bigr)_{m} (q\alpha; q)_{m+n}}\\
&\qquad{}=\frac{\bigl(q^{-N}, q/c, q/d; q\bigr)_n q^n}{(q\alpha; q)_{2n} \bigl(q^2/{\alpha cd q^N}; q\bigr)_n}
\sum_{j=0}^{N-n} \frac{\bigl(q^{-N+n}, q^{n+1}/c, q^{n+1}/d; q\bigr)_j q^{j}}
{\bigl(q, q^2/{\alpha cd q^{N-n}}, \alpha q^{2n+1}; q\bigr)_j}\nonumber\\	
&\qquad{}=\frac{\bigl(q^{-N}, q/c, q/d; q\bigr)_n (\alpha cq^n, \alpha dq^n; q)_{N-n} q^n}{(q\alpha; q)_{2n} \bigl(q^2/{\alpha cd q^N}; q\bigr)_n \bigl(\alpha q^{2n+1}, \alpha cd/q; q\bigr)_{N-n}}\nonumber\\
&\qquad{}=\frac{\bigl(q^{-N}, q/c, q/d; q\bigr)_n (\alpha c, \alpha d; q)_N}
{(q\alpha; q)_{n+N} (\alpha c, \alpha d; q)_n (\alpha cd/q; q)_N}
\bigl(-\alpha cd q^N\bigr)^n q^{-n(n+1)/2}. \nonumber
\end{align*}	
Substituting the above equation into (\ref{expliu:eqn9}) and then interchanging $(a, b)$ and $(c, d)$, and then change $n$ to $k$, and finally change $N$ to $n$, we complete the proof of Watson's $q$-analogue of Whipple's theorem.
\end{proof}

\subsection[Evaluation of some terminating q-series]{Evaluation of some terminating $\boldsymbol{q}$-series}
This subsection is devoted to the evaluation of some terminating $q$-series.

Taking $q/a=\alpha q^{n+1}$ in Theorem~\ref{WatsonWhipplethm} and simplifying, we find that \cite[equation~(2.1)]{Liu2013IJTN}
\begin{align}
	&{_4\phi_3}\biggl({{q^{-n}, \alpha q^{n+1}, q/b, \alpha cd/q }\atop{\alpha c, \alpha d, q^2/b}}; q, q\biggr)\nonumber\\
	&\qquad{}=\frac{(q, \alpha b; q)_n} {\bigl(\alpha q, q^2/b; q\bigr)_n } \Bigl(\frac{q} {b}\Bigr)^n \sum_{j=0}^n \frac{\bigl(1-\alpha q^{2j}\bigr) (\alpha, q/b, q/c, q/d; q)_j}
	{(1-\alpha)(q, \alpha b, \alpha c, \alpha d; q)_j}\biggl(\frac{\alpha bcd}{q^2}\biggr)^{j}.\label{newvalue:eqn1}
\end{align}

Letting $\alpha \to 1$ in the above equation and making a direct computation, we deduce that
\begin{align}\label{newvalue:eqn2}
	&{_4\phi_3}\biggl({{q^{-n}, q^{n+1}, q/b, cd/q }\atop{c, d, q^2/b}}; q, q\biggr)\\
	&\qquad{}=\frac{(b; q)_n} {\bigl(q^2/b; q\bigr)_n } \Bigl(\frac{q} {b}\Bigr)^n
	\Biggl(1+\sum_{j=1}^n \bigl(1+q^{j}\bigr)\frac{(q/b, q/c, q/d; q)_j}
	{(b, c, d; q)_j}\biggl(\frac{ bcd}{q^2}\biggr)^{j}\Biggr).\nonumber
\end{align}
Letting $b=-q$ in the equation above and then putting $c=d=0$ in the resulting equation yields%
\begin{equation}\label{newvalue:eqn3}
{_3\phi_2}\biggl({{q^{-n}, q^{n+1}, -1}\atop{0, -q}}; q, q\biggr)
=(-1)^n \sum_{j=-n}^n (-1)^j q^{j^2}.
\end{equation}
Taking $b=d=0$ and $c=-q$ in (\ref{newvalue:eqn2}), it is found that \cite[p.~2088]{Liu2013IJTN}
\begin{equation}\label{newvalue:eqn4}
{_2\phi_1}\biggl({{q^{-n}, q^{n+1}}\atop{ -q}}; q, 1\biggr)
=(-1)^n q^{-n(n+1)/2} \sum_{j=-n}^n (-1)^j q^{j^2}.
\end{equation}
(The factor $(-1)^n$ is missing in \cite[Propositions 2.4--2.6 and Theorem~4.10]{Liu2013IJTN}.)

Replacing $q$ by $q^2$ in (\ref{newvalue:eqn2}) and then putting $b=-q^2$, $c=-q$ and $d=0$ in the above equation, we immediately deduce that
\begin{equation}\label{newvalue:eqn5}
{_3\phi_2}\biggl({{q^{-2n}, q^{2n+2}, -1}\atop{-q, -q^2}}; q^2, q^2\biggr)=(-1)^n \sum_{j=-n}^n (-1)^j q^{j^2}.
\end{equation}	
 Replacing $q$ by $q^2$ in (\ref{newvalue:eqn2}) and then setting $b=-q$, $c=-q^2$ and $d=0$ in the resulting equation, we arrive at
 \begin{equation*}%\label{newvalue:eqn6}
 	{_3\phi_2}\biggl({{q^{-2n}, q^{2n+2}, -q}\atop{-q^2, -q^3}}; q^2, q^2\biggr)=(-q)^n \biggl(\frac{1+q}{1+q^{2n+1}}\biggr)\sum_{j=-n}^n (-1)^j q^{j^2}.
 \end{equation*}	
Writing $q$ by $q^2$ in (\ref{newvalue:eqn2}) and then
 taking $b=0$, $c=-q^2$ and $d=-q$, we conclude that \mbox{\cite[equa\-tion~(2.7)]{WangYee2019}}
 \begin{equation*}%\label{newvalue:eqn7}
 	{_3\phi_2}\biggl({{q^{-2n}, q^{2n+2}, q}\atop{-q, -q^2}}; q^2, 1\biggr)=(-1)^n q^{-n(n+1)}\sum_{j=-n}^n (-1)^j q^{j^2}.
 	\end{equation*}

 Letting $b\to \infty$ in (\ref{newvalue:eqn1}), we arrive at the following theorem \cite[Proposition~2.3]{Liu2013IJTN}.
 \begin{Theorem}\label{spwwthm} The following transformation formula for terminating $q$-series holds:
 	\begin{align}
 	&{_3\phi_2}\biggl({{q^{-n}, \alpha q^{n+1}, \alpha cd/q }\atop{\alpha c, \alpha d}}; q, q\biggr)\nonumber\\
 	&\qquad{}=(-\alpha)^n q^{n(n+1)/2}\frac{(q; q)_n}{(q\alpha; q)_n}
 	\sum_{j=0}^n (-1)^j\frac{\bigl(1-\alpha q^{2j}\bigr) (\alpha, q/c, q/d; q)_j}
 	{(1-\alpha)(q, \alpha c, \alpha d; q)_j}\biggl(\frac{ cd}{q}\biggr)^{j} q^{-j(j+1)/2}.\label{expliu:eqn11}
 \end{align}
 \end{Theorem}	
Letting $b \to \infty$ and $d\to \infty$ in Theorem~\ref{WatsonWhipplethm} and then setting $\alpha c=q^{-n}$,
we arrive at the following theorem of Andrews's theorem \cite[equation~(4.6)]{Andrews86}, which includes Shank's finite version of Euler's pentagonal number theorem \cite{Shanks1951} and Shank's finite version of Gauss's theorem~\cite{Shanks1958} as special cases.
\begin{Theorem} %\label{AndrewsShanks}
We have
\begin{align}\label{expliu:eqn12}
\frac{(q\alpha; q)_n}{(\alpha a; q)_n}\biggl(\frac{a}{q}\biggr)^n
	\sum_{j=0}^n \frac{(q/a; q)_j (\alpha q^n)^{-j}}{(q; q)_j}
=\sum_{j=0}^n \frac{\bigl(1-\alpha q^{2j}\bigr) (\alpha, q/a; q)_j (a/{\alpha})^j q^{-j^2-j}}{(1-\alpha)(q, \alpha a; q)_j}.
\end{align}	
\end{Theorem}

\begin{Theorem}\label{andrewsliu} We have
\begin{align*}%\label{spvalue:eqn1}
&{_3\phi_2}\biggl({{q^{-n}, \alpha q^{n+1},q}\atop{q^2/a, 0}}; q, q\biggr)\\
&\qquad{}=q^{n^2+2n} \frac{(q, \alpha a; q)_n}{\bigl(q\alpha, q^2/a; q\bigr)_n} \Bigl(\frac{\alpha}{a}\Bigr)^n \sum_{j=0}^n \frac{\bigl(1-\alpha q^{2j}\bigr) (\alpha, q/a; q)_j (a/{\alpha})^j q^{-j^2-j}}{(1-\alpha)(q, \alpha a; q)_j}.\nonumber
\end{align*}	
\end{Theorem}
\begin{proof}
 With the help of the finite $q$-binomial theorem and the $q$-Chu--Vandermonde summation, we \cite[Lemma~4.1]{Liu2013RamanJ} proved that
	\begin{align*}%\label{spvalue:eqn2}
	{_3\phi_2}\biggl({{q^{-n}, \alpha q^n,q}\atop{qc, 0}}; q, q\biggr)
	=\alpha^n q^{n^2}\frac{(q; q)_n}{(qc; q)_n}
		\sum_{j=0}^n \frac{(c; q)_j \alpha^{-j} q^{j(1-n)}}{(q; q)_j}.
	\end{align*}
Replacing $\alpha$ by $q\alpha$ and $c$ by $q/a$ in the equation above and then comparing the resulting equation with (\ref{expliu:eqn12}), we complete the proof of Theorem~\ref{andrewsliu}.
\end{proof}

Setting $a=-q$ in Theorem~\ref{andrewsliu} and then letting $\alpha \to 1$, we arrive at the following identity, which is equivalent to
\cite[equation~(6.1)]{Liu2013RamanJ}:
\begin{equation}\label{spvalue:eqn3}
{_3\phi_2}\biggl({{q^{-n}, q^{n+1},q}\atop{-q, 0}}; q, q\biggr)
=(-1)^n q^{n^2+n}\sum_{j=-n}^n (-1)^j q^{-j^2}.
\end{equation}

Letting $\alpha \to 1$ and $a \to \infty$ in Theorem~\ref{andrewsliu}, we get the following identity, which is equivalent to \cite[equation~(6.2)]{Liu2013RamanJ}:
\begin{equation*}%\label{spvalue:eqn4}
	{_3\phi_2}\biggl({{q^{-n}, q^{n+1},q}\atop{0, 0}}; q, q\biggr)
	=(-1)^n q^{3n(n+1)/2}\sum_{j=-n}^n (-1)^j q^{-{(3j^2+j)/2}}.
\end{equation*}

Taking $a=-q$ in Theorem~\ref{andrewsliu} and then letting $\alpha \to q^{-1}$ and making some calculation, we deduce the following identity, which is equivalent to the first identity of \cite[Lemma~6.1]{Liu2013RamanJ}:
\begin{align}\label{spvalue:eqn5}
	&{_3\phi_2}\biggl({{q^{-n}, q^{n},q}\atop{-q, 0}}; q, q\biggr)
=(-1)^n\frac{q^{n^2}}{1+q^n} \Biggl(q^n \sum_{j=-n}^n (-1)^j q^{-j^2}-\sum_{j={-n+1}}^{n-1} (-1)^j q^{-j^2}\Biggr).
\end{align}

Letting $\alpha \to q^{-1}$ and $a \to \infty$ in Theorem~\ref{andrewsliu} and making some calculation, we deduce the following identity, which is equivalent to the second identity of \cite[Lemma~6.1]{Liu2013RamanJ}:
\begin{align}
	&{_3\phi_2}\biggl({{q^{-n}, q^{n},q}\atop{0, 0}}; q, q\biggr)\nonumber\\
	&\qquad{}=(-1)^n\frac{q^{(3n^2-n)/2}}{1+q^n} \Biggl({ q^{2n}} \sum_{j=-n}^n (-1)^j q^{-(3j^2+j)/2}-\sum_{j={-n+1}}^{n-1} (-1)^j q^{-(3j^2+j)/2}\Biggr).\label{spvalue:eqn6}
\end{align}

Replacing $q$ by $q^2$ in Theorem~\ref{andrewsliu} and then letting $\alpha \to 1$ and setting $a=q$, we find that
\begin{align}
{_3\phi_2}\biggl({{q^{-2n}, q^{2n+2},q^2}\atop{q^3, 0}}; q^2, q^2\biggr)
&{}=\frac{(1-q)}{\bigl(1-q^{2n+1}\bigr)} q^{2n^2+3n}\sum_{j=-n}^{n} q^{-2j^2-j}	\nonumber\\
&{}={ \frac{(1-q)}{\bigl(1-q^{2n+1}\bigr)} q^{2n^2+3n}\sum_{j=0}^{2n} q^{-j(j+1)/2}.}\label{spvalue:eqn7}
\end{align}	

\begin{Theorem} \label{watliuthm} The following $q$-transformation formula holds:
	\begin{align*}%\label{spvalue:eqn8}
 \sum_{k=0}^n \frac{(q^{-n}, q^n; q)_k q^k}{(q/a, q/b; q)_k}
		&{}=(-1)^n q^{n(n+1)/2} (1-q^n) \frac{(a, b; q)_n}{(q/a, q/b; q)_n}\biggl(\frac{1}{ab}\biggr)^n\nonumber\\
		&\quad{} \times \Biggl\{\frac{1-ab}{(1-a)(1-b)}+\frac{(1/a, 1/b; q)_n}
		{(1-q^n)(a, b; q)_n} (-ab)^n q^{-n(n-1)/2} \nonumber\\
		&\qquad\quad{}+(1-ab)\sum_{j=1}^{n-1} \frac{(1+q^j) (1/a, 1/b; q)_j}{ (a, b; q)_{j+1}} (-ab)^j q^{-j(j-1)/2}\Biggr\}.\nonumber
	\end{align*}		
\end{Theorem}
\begin{proof}
	Setting $\beta=q$ and $\gamma=0$ in the Sears $_4\phi_3$ transformation in Proposition~\ref{Searstrs}, we deduce that
	\begin{equation*}%\label{spvalue:eqn9}
		{_3 \phi_2} \biggl({{q^{-n}, \alpha q^n, \alpha ab}
			\atop{\alpha a, \alpha b}}; q, q\biggr)
		=\frac{(q/a, q/b; q)_n}{(\alpha a, \alpha b; q)_n}
		\biggl(\frac{\alpha ab}{q}\biggr)^n
		\sum_{k=0}^n \frac{(q^{-n}, \alpha q^n; q)_k q^k}{(q/a, q/b; q)_k}.
	\end{equation*}
Replacing $\alpha$ by $\alpha/q$ and $(c, d)$ by $(qa, qb)$ in Theorem~\ref{spwwthm}, we conclude that
\begin{align*}%\label{spvalue:eqn10}
	\sum_{k=0}^n \frac{(q^{-n}, \alpha q^n; q)_k q^k}{(q/a, q/b; q)_k}
	&{}=(-1)^n q^{n(n+1)/2} \frac{(q, \alpha a, \alpha b; q)_n}{(\alpha, q/a, q/b; q)_n}\biggl(\frac{1}{ab}\biggr)^n\\
	&\quad{}\times \sum_{j=0}^n \frac{\bigl(1-\alpha q^{2j-1}\bigr) (\alpha/q, 1/a, 1/b; q)_j}{(1-\alpha/q) (q, \alpha a, \alpha b; q)_j} (-ab)^j q^{-j(j-1)/2}.\nonumber
\end{align*}	
Combining the above two equations, we complete the proof of Theorem~\ref{watliuthm}.
\end{proof}

Replacing $q$ by $q^2$ in Theorem~\ref{watliuthm} and then taking $a=-1$ and $b=-q$, we can deduce that
\begin{align}
 &(1+q^{2n})	\sum_{k=0}^n \frac{\bigl(q^{-2n}, q^{2n}; q^2\bigr)_k q^{2k}}{\bigl(-q, -q^2; q^2\bigr)_k}\nonumber\\
&\qquad{}=(-1)^n q^{n^2}\Biggl(q^{2n}\sum_{j=-n}^n (-1)^j q^{-j^2}-\sum_{j=-n+1}^{n-1}(-1)^j q^{-j^2}\Biggr).\label{spvalue:eqn11}
\end{align}

Setting $a=q^{1/2}$ and $b=-q^{1/2}$ in Theorem~\ref{watliuthm} and simplifying, we arrive at \cite[equation~(5.3)]{Andrews2012}%
\begin{equation*}%\label{spvalue:eqn12}
	{_3\phi_2}\left({{q^{-n}, q^{n}, q }\atop{q^{1/2}, -q^{1/2}}}; q, q\right)
	= q^{n(n-1)/2}\left(q^n \sum_{j=0}^n q^{-j(j+1)/2}
	-\sum_{j=0}^{n-1} q^{-j(j+1)/2}\right).
\end{equation*}

\section[The q-exponential differential operator and the Sears \_4phi\_3 transformation]{The $\boldsymbol{q}$-exponential differential operator\\ and the Sears $\boldsymbol{_4\phi_3}$ transformation}\label{sec3}

\subsection[Some basic properties of the q-exponential differential operator]{Some basic properties of the $\boldsymbol{q}$-exponential differential operator}

Now we give the definitions of the $q$-partial derivative and $q$-partial differential equations \cite{LiuHahn2015}.
{ \begin{Definition}%\label{qpdfn}
Given a multivariable function $f(x_1, x_2, \dots, x_n)$, the $q$-partial derivative of $f$ with respect to $x_j$, $j=1, 2, \ldots, n$, is its $q$-derivative with respect
$x_j$, regarding other variables as constants, which is denoted by
 $\partial_{q, x_j}f$. The formula for the $q$-partial derivative of $f$ is~given~by
\begin{equation*}
\partial_{q, x_j}f=\frac{f(x_1, \ldots, x_{j}, \ldots, x_n)-f(x_1, \ldots, x_{j-1}, qx_{j}, x_{j+1}, \ldots, x_n)}{x_j}.
\end{equation*}
\end{Definition}}
\begin{Definition}%\label{qpde}
A $q$-partial differential equation is an equation that contains unknown multivariable functions and their $q$-partial derivatives.
\end{Definition}
	
The concept of $q$-partial differential equations is quite useful in $q$-analysis \cite{AslanIsmail, LiuRam2013, LiuHahn2015, Liu2017}.
Based on $\partial_{q, x}$, we can construct the $q$-exponential differential operator $T(y\partial_{q, x})$ as follows:
\begin{align*}%\label{qedo:eqn1}
	T(y\partial_{q, x})=\sum_{n=0}^\infty \frac{(y \partial_{q, x})^n}{(q; q)_n}.
\end{align*}

One of the most important results of the $q$-exponential differential operational operator is the following operator identity \cite[equation~(3.1)]{Liu2010}.
\begin{Proposition}\label{liuopeppb}For $\max\{|as|, |at|, |au|, |bs|, |bt|, |bu|, |abstu/v|\}<1, $	we have
	\begin{align*}
		&T(b \partial_{q, a})\biggl\{\frac{(av; q)_\infty}{(as, at, au; q)_\infty}\biggr\}
=\frac{(av, bv, abstu/v; q)_\infty}{(as, at, au, bs, bt, bu; q)_\infty}
 {_3\phi_2}\biggl({{v/s, v/t, v/u}\atop{av, bv}}; q, \frac{abstu}{v}\biggr).
	\end{align*}
\end{Proposition}
The following proposition can be found in \cite[Lemma~13.2]{LiuRam2013} and its proof is very short. For completeness, we will repeat the proof here.
\begin{Proposition}\label{liulemope} Suppose that $f(x, y)$ is a two-variable analytic function near $(x, y)=(0, 0).$
	Then
	$f(x, y)=T(y \partial_{q, x})f(x, 0)$ if and only if $\partial_{q, x}f=\partial_{q, y}f$.
\end{Proposition}

\begin{proof}
	If $f(x, y)=T(y \partial_{q, x})f(x, 0)$, then by a direct computation we find that
	$\partial_{q, x}f=\partial_{q, y}f=T(y \partial_{q, x})\partial_{q, x}f(x, 0).$
	Conversely, if $f$ is a two-variable analytic function near $(x, y)=(0, 0)$, satisfying
	$\partial_{q, x}f=\partial_{q, y}f,$ then we may assume that
	\[
	f(x, y)=\sum_{n=0}^\infty A_n(x) \frac{y^n}{(q; q)_n}.
	\]
	Substituting the equation into the $q$-partial differential equation
	$\partial_{q, x}f=\partial_{q, y}f$, we easily find~that
	\[
	A_n(x)=\partial_{q, x}A_{n-1}(x)=\cdots=\partial^n_{q, x}A_0(x).
	\]
	It is easily seen that $A_0(x)=f(x, 0).$ Thus, we conclude that
	\[
	f(x, y)=\sum_{n=0}^\infty \partial^n_{q, x}f(x, 0)
	\frac{y^n}{(q; q)_n}
	=T(y \partial_{q, x})f(x, 0).
	\]
	This completes the proof of Proposition~\ref{liulemope}.
\end{proof}

It is proved \cite[Proposition~13.4]{LiuRam2013} that if $f(x)$ is analytic near $x=0$, then $T(y \partial_{q, x})f(x)$ is analytic near $(x, y)=(0, 0)$.

Using Proposition~\ref{liulemope}, we \cite[Proposition~13.9]{LiuRam2013} proved the following proposition, which allows us to act an infinite series term by term with $T(y\partial_{q, x})$.
\begin{Proposition} \label{liuopeppf}Let $\{f_n(x)\}$ be a sequence of analytic functions near $x=0$
	such that the series
	\[
	\sum_{n=0}^\infty f_n(x)
	\]
	converges uniformly to an analytic function $f(x)$ near $x=0$, and the series
	\[
	\sum_{n=0}^\infty T(y\partial_{q, x})f_n(x)
	\]
	converges uniformly to an analytic function $f(x, y)$ near $(x, y)=(0, 0)$. Then we have
	$f(x, y)=T(y\partial_{q, x})f(x)$, or
	\[
	T(y\partial_{q, x})\Biggl\{\sum_{n=0}^\infty f_n(x)\Biggr\}=\sum_{n=0}^\infty T(y\partial_{q, x})f_n(x).
	\]
\end{Proposition}
\begin{proof} If we use $f_n(x, y)$ to denote $T(y\partial_{q, x})f_n(x), $ then by Proposition~\ref{liulemope}, we find that $\partial_{q, x}f_n(x, y)=\partial_{q, y}f_n(x, y)$. It follows that
	$\partial_{q, x}f(x, y)=\partial_{q, y}f(x, y)$. Thus, using Proposition~\ref{liulemope} again,
	we have
	$
	f(x, y)=T(y\partial_{q, x})f(x, 0)=T(y\partial_{q, x})f(x).
	$
\end{proof}

A multivariable form of Proposition~\ref{liuopeppf} is proved in \cite{Liu2023Acta}, and a multiple $q$-exponential differential
operator identity is derived in \cite{Liu2023ActaSci}.

\subsection[An extension of the Sears \_4phi\_3 transformation formula]{An extension of the Sears $\boldsymbol{_4\phi_3}$ transformation formula}

One of the fundamental formulas in the theory of $q$-series is the Sears $_4\phi_3$ transformation formula \cite[p.~49, equation~(2.10.4)]{Gas+Rah},
which is stated in the following proposition.
\begin{Proposition}[the Sears $_4\phi_3$ transformation formula]\label{Searstrs}
	We have
	\begin{align*}
		&{_4 \phi_3} \biggl({{q^{-n}, \alpha q^n, \alpha \beta ab/q, \alpha \gamma ab/q}
			\atop{\alpha a, \alpha b,\alpha \beta \gamma ab/q}}; q, q\biggr)\\
		&\qquad{}=\frac{(q/a, q/b; q)_n}{(\alpha a, \alpha b; q)_n}
		\biggl(\frac{\alpha ab}{q}\biggr)^n
		{_4 \phi_3} \biggl({{q^{-n}, \alpha q^n, \beta, \gamma}
			\atop{q/a, q/b,\alpha \beta \gamma ab/q}}; q, q\biggr).
	\end{align*}
\end{Proposition}
In this paper, we will extend the Sears $_4\phi_3$ transformation formula to the following theorem by using Propositions~\ref{liuopeppb} and ~\ref{liuopeppf}.
\begin{Theorem}\label{ExtSearstrs} The following transformation formula
	for double $q$-series holds:
	\begin{align*}
		&\sum_{k=0}^{n} \frac{(q^{-n}, \alpha q^n, \alpha \beta ab/q, \alpha \gamma ab/q ; q)_k q^k}
		{(q, \alpha a, \alpha b, \beta \gamma uv; q)_k} A_k\\
		&\qquad{}=\frac{(q/a, q/b; q)_n}{(\alpha a, \alpha b; q)_n}
		\biggl(\frac{\alpha ab}{q}\biggr)^n
		\sum_{k=0}^{n} \frac{(q^{-n}, \alpha q^n, \beta, \gamma; q)_k q^k}{(q, q/a, q/b, \beta \gamma uv; q)_k}B_k,
	\end{align*}
	where $A_k$ and $B_k$ are given by
	\begin{equation*}
		A_k= {_3\phi_2} \biggl({{q^{-k}, \alpha ab/qu, \alpha ab/qv}
			\atop{\alpha \beta ab/q, \alpha \gamma ab/q }}; q, \beta \gamma uv q^k \biggr),\qquad
		B_k={_3\phi_2 \biggl({{q^{-k}, 1/u, 1/v}\atop{\beta, \gamma}}; q, \beta \gamma uv q^k\biggr)}.
	\end{equation*}
\end{Theorem}
Note that when $u=1,$ $B_k$ reduces to $1$ and when $v=\alpha ab/q$, $A_k$ becomes $1$; thus we
can take $u=1$ and $v=\alpha ab/q$ in Theorem~\ref{ExtSearstrs} to obtain the Sears $_4\phi_3$ transformation formula. Hence, Theorem~\ref{ExtSearstrs} is an extension of the Sears $_4\phi_3$ transformation formula.
Now we prove Theorem~~\ref{ExtSearstrs}.
\begin{proof}
Setting $\gamma=0$ in the Sears $_4\phi_3$ transformation in Proposition~\ref{Searstrs}, we obtain the following transformation:
\begin{align*}
{_3 \phi_2} \biggl({{q^{-n}, \alpha q^n, \alpha \beta ab/q}
\atop{\alpha a, \alpha b}}; q, q\biggr)
=\frac{(q/a, q/b; q)_n}{(\alpha a, \alpha b; q)_n}
\biggl(\frac{\alpha ab}{q}\biggr)^n
{_3 \phi_2} \biggl({{q^{-n}, \alpha q^n, \beta}
\atop{q/a, q/b}}; q, q\biggr).
\end{align*}
Using the identity, $(z; q)_\infty=(z; q)_k \bigl(zq^k; q\bigr)_\infty$, we can rewrite this equation as
\begin{align*}
&\sum_{k=0}^{n} \frac{(q^{-n}, \alpha q^n; q)_k q^k}{(q, \alpha a, \alpha b; q)_k}
\frac{(\alpha \beta ab/q; q)_\infty}{(\alpha \beta ab q^k/q; q)_\infty} \\
&\qquad{}=\frac{(q/a, q/b; q)_n}{(\alpha a, \alpha b; q)_n}
\biggl(\frac{\alpha ab}{q}\biggr)^n
\sum_{k=0}^{n} \frac{(q^{-n}, \alpha q^n; q)_k q^k}{(q, q/a, q/b; q)_k}
\frac{(\beta; q)_\infty}{(\beta q^k; q)_\infty}.
\end{align*}
Multiplying both sides of the equation by $1/(\beta u, \beta v; q)_\infty$, we arrive at
\begin{align*}
&\sum_{k=0}^{n} \frac{(q^{-n}, \alpha q^n; q)_k q^k}{(q, \alpha a, \alpha b; q)_k}
\frac{(\alpha \beta ab/q; q)_\infty}{(\alpha \beta ab q^k/q, \beta u, \beta v; q)_\infty} \\
&\qquad{}=\frac{(q/a, q/b; q)_n}{(\alpha a, \alpha b; q)_n}
\biggl(\frac{\alpha ab}{q}\biggr)^n
\sum_{k=0}^{n} \frac{(q^{-n}, \alpha q^n; q)_k q^k}{(q, q/a, q/b; q)_k}
\frac{(\beta; q)_\infty}{(\beta q^k, \beta u, \beta v; q)_\infty}.
\end{align*}
Both sides of this equation are finite series, and each term of these two series is analytic at~${\beta=0}$.
Hence, we can apply $T(\gamma \partial_{q, \beta})$ to both sides of the above equation
to obtain
\begin{align}
&\sum_{k=0}^{n} \frac{(q^{-n}, \alpha q^n; q)_k q^k}{(q, \alpha a, \alpha b; q)_k}
T(\gamma \partial_{q, \beta}) \biggl\{\frac{(\alpha \beta ab/q; q)_\infty}{\bigl(\alpha \beta ab q^k/q, \beta u, \beta v; q\bigr)_\infty} \biggr\}\nonumber\\
&\qquad{}=\frac{(q/a, q/b; q)_n}{(\alpha a, \alpha b; q)_n}
\biggl(\frac{\alpha ab}{q}\biggr)^n \sum_{k=0}^{n} \frac{(q^{-n}, \alpha q^n; q)_k q^k}{(q, q/a, q/b; q)_k}
T(\gamma \partial_{q, \beta})\biggl\{\frac{(\beta; q)_\infty}{\bigl(\beta q^k, \beta u, \beta v; q\bigr)_\infty}\biggr\}.
\label{extsears:eqn1}
\end{align}
With the help of the $q$-exponential operator identity in Proposition~\ref{liuopeppb}, we have
\begin{align*}
& T(\gamma \partial_{q, \beta}) \biggl\{ \frac{(\alpha \beta ab/q; q)_\infty}{\bigl(\alpha \beta ab q^k/q, \beta u, \beta v; q\bigr)_\infty} \biggr\}\\
&\qquad{}=\frac{\bigl(\alpha \beta ab/q, \alpha \gamma ab/q, \beta \gamma uvq^k; q\bigr)_\infty}
{\bigl(\alpha \beta ab q^k/q, \beta u, \beta v, \alpha \gamma ab q^k/q, \gamma u, \gamma v; q\bigr)_\infty}
 {_3\phi_2} \biggl({{q^{-k}, \alpha ab/qu, \alpha ab/qv}
\atop{\alpha \beta ab/q, \alpha \gamma ab/q }}; q, \beta \gamma uv q^k \biggr)
\end{align*}
and
\begin{align*}
& T(\gamma \partial_{q, \beta}) \biggl\{ \frac{(\beta; q)_\infty}{\bigl(\beta q^k, \beta u, \beta v; q\bigr)_\infty}
 \biggr\}\\
&\qquad{}=\frac{\bigl(\beta, \gamma, \beta \gamma uvq^k; q\bigr)_\infty}
{\bigl(\beta q^k, \beta u, \beta v, \gamma q^k, \gamma u, \gamma v; q\bigr)_\infty}
{_3\phi_2 \biggl({{q^{-k}, 1/u, 1/v}\atop{\beta, \gamma}}; q, \beta \gamma uv q^k\biggr)}.
\end{align*}
Substituting these two equations into (\ref{extsears:eqn1}) and simplifying, we complete
the proof of Theorem~\ref{ExtSearstrs}.
\end{proof}

\subsection{The proof of Theorem~\ref{liuthma}}

\begin{proof}
	For convenience, we replace $a$ by $c$ in Theorem~\ref{liuthm} to conclude that
	\begin{align}\label{liurogers:eqn1}
	f(\alpha c)
	=\sum_{n=0}^\infty \frac{\bigl(1-\alpha q^{2n}\bigr)(\alpha, q/c; q)_n (c/q)^n}
	{(1-\alpha)(q, \alpha c; q)_n} \sum_{k=0}^n \frac{(q^{-n}, \alpha q^n; q)_k q^k}
	{(q, q\alpha; q)_k} f\bigl(\alpha q^{k+1}\bigr).
\end{align}
Using the ratio test, we can easily show that the following function is an analytic function of
$x$ at $x=0$:
\begin{align*}
\frac{(ax/q, bx/q, \beta \gamma uvx/q\alpha; q)_\infty}
{\bigl(x, \beta abx/q^2, \gamma abx/q^2; q\bigr)_\infty}
{_3\phi_2}\biggl({{q\alpha/x, \alpha ab/qu, \alpha ab/qv}\atop{\alpha \beta ab/q, \alpha \gamma ab/q}}; q, \frac{\beta \gamma uvx}{q\alpha}\biggr).
\end{align*}
Hence, we can replace $f(x)$ by this function in (\ref{liurogers:eqn1}). It is easily seen that
\begin{align*}
f(\alpha c) = \frac{(\alpha ac/q, \alpha bc/q, \beta \gamma uvc/q; q)_\infty}
{\bigl(\alpha c, \alpha \beta abc/q^2, \alpha \gamma abc/q^2; q\bigr)_\infty} {_3\phi_2}\biggl({{q/c, \alpha ab/qu, \alpha ab/qv}\atop{\alpha \beta ab/q, \alpha \gamma ab/q}}; q, \frac{\beta \gamma uvc}{q}\biggr)
\end{align*}
and
\begin{align*}
f\bigl(\alpha q^{k+1}\bigr)&=\frac{(\alpha a, \alpha b, \beta \gamma uv; q)_\infty (q\alpha, \alpha \beta ab/q, \alpha \gamma ab/q; q)_k}
{(q\alpha, \alpha \beta ab/q, \alpha \gamma ab/q; q)_\infty (\alpha a, \alpha b, \beta \gamma uv; q)_k}\\
&\quad \times {_3\phi_2} \biggl({{q^{-k}, \alpha ab/qu, \alpha ab/qv}
\atop{\alpha \beta ab/q, \alpha \gamma ab/q }}; q, \beta \gamma uv q^k \biggr).
\end{align*}
It follows that
\begin{align*}
& \frac{(\alpha {q}, \alpha ac/q, \alpha bc/q, \alpha \beta ab/q, \alpha \gamma ab/q, \beta \gamma cuv/q; q )_\infty}{\bigl(\alpha a, \alpha b, \alpha c, \alpha \beta abc/q^2, \alpha \gamma abc/q^2, \beta \gamma uv; q \bigr)_\infty}
{_3\phi_2} \biggl({{q/c, \alpha ab/qu, \alpha ab/qv}\atop{\alpha \beta ab/q, \alpha \gamma ab/q}}; q,
\frac{\beta \gamma cuv}{q} \biggr)\\
&\qquad{}=\sum_{n=0}^\infty \frac{\bigl(1-\alpha q^{2n}\bigr) (\alpha, q/c; q)_n (c/q)^n}{(1-\alpha)(q, \alpha c; q)_n}
\sum_{k=0}^{n} \frac{(q^{-n}, \alpha q^n, \alpha \beta ab/q, \alpha \gamma ab/q ; q)_k q^k}
{(q, \alpha a, \alpha b, \beta \gamma uv; q)_k}\\
&\qquad\hphantom{=\sum_{n=0}^\infty}{}
\times {_3\phi_2} \biggl({{q^{-k}, \alpha ab/qu, \alpha ab/qv}
\atop{\alpha \beta ab/q, \alpha \gamma ab/q }}; q, \beta \gamma uv q^k \biggr).
\end{align*}
Using the extension of the Sears $_4\phi_3$ transformation in Theorem~\ref{ExtSearstrs} to
the right-hand side of the above equation, we complete the proof of the theorem.
\end{proof}

\section{Some applications of Theorem~\ref{rogersliuthm} in number theory}\label{sec4}
\subsection{Some preliminaries}
Setting $c=0$ and letting $\alpha \to 1$ in Proposition~\ref{rogersthm}, we find that
\begin{equation} \label{pre:eqn1}
	\frac{(q, ab/q; q)_\infty}{(a, b; q)_\infty}
	=1+\sum_{n=1}^\infty (-1)^n (1+q^n) (ab)^n q^{n(n-3)/2} \frac{(q/a, q/b; q)_n}{(a, b; q)_n}.
\end{equation}
Setting $a=b=0$ in the above equation, we immediately arrive at Euler's pentagonal theo\-rem~\mbox{\cite[equation~(15)]{Gas+Rah}}%
\begin{equation*}%\label{pre:eqn2}
	(q; q)_\infty=1+\sum_{n=1}^\infty (-1)^n (1+q^n) q^{(3n^2-n)/2}.
\end{equation*}
Putting $b=0$ and $a=-q$ in (\ref{pre:eqn1}) and simplifying, we arrive at the following identity due to Gauss \cite[p.~347]{Liu2010pacific}:
\begin{equation}\label{pre:eqn3}
	\frac{(q; q)_\infty}{(-q; q)_\infty}
	=1+2\sum_{n=1}^\infty (-1)^n q^{n^2}
	=\sum_{n=-\infty}^\infty (-1)^n q^{n^2}.
\end{equation}
Replacing $q$ by $q^2$ in (\ref{pre:eqn1}) and then putting $a=q$ and $b=0$, we are led to the Gauss iden\-tity~\mbox{\cite[p.~347]{Liu2010pacific}}
\begin{equation}\label{pre:eqn4}
	\frac{\bigl(q^2; q^2\bigr)_\infty}{\bigl(q; q^2\bigr)_\infty}
	=1+\sum_{n=1}^\infty \bigl(1+q^{2n}\bigr) q^{2n^2-n}
	=\sum_{n=0}^\infty q^{n(n+1)/2}.
\end{equation}

Applying the Sears $_4\phi_3$ transformation formula in Proposition~\ref{Searstrs}
to Theorem~\ref{rogersliuthm}, we immediately obtain the following summation formula for $q$-series.
\begin{Theorem}\label{prodliuthm}
	For $\max\bigl\{\bigl|\alpha \beta abc/q^2\bigr|, \bigl|\alpha \gamma abc/q^2\bigr|\bigr\}<1$, we have
	\begin{align*}
		&\sum_{n=0}^\infty \frac{\bigl(1-\alpha q^{2n}\bigr)(\alpha, q/c; q)_n (c/q)^n}{(1-\alpha)(q, \alpha c; q)_n}~ {_4 \phi_3} \biggl({{q^{-n}, \alpha q^n, \alpha \beta ab/q, \alpha \gamma ab/q}
			\atop{\alpha a, \alpha b,\alpha \beta \gamma ab/q}}; q, q\biggr)\\
		&\qquad{}=\frac{\bigl(q\alpha, \alpha ac/q, \alpha bc/q, \alpha \beta ab/q, \alpha \gamma ab/q, \alpha \beta \gamma abc/q^2; q\bigr)_\infty}
		{\bigl(\alpha a, \alpha b, \alpha c, \alpha \beta abc/q^2, \alpha \gamma abc/q^2, \alpha \beta \gamma ab/q; q \bigr)_\infty}.
	\end{align*}
\end{Theorem}
Next, we will discuss some applications of this formula in number theory and Rogers--Hecke type series.
%%%%%%%%%%%%%%%%%%%%%%%%%%%%%%%%%%%%%%%%%%%%%%%%%%%%%%%%%%%%%%%%%%%%%%%%%%%%
\subsection{Andrews' identity for the sums of three squares}
Setting $\beta=q^2/ab$, $\gamma=0$ and letting $\alpha \to 1$ in Theorem~\ref{prodliuthm}, we conclude that
\begin{align}\label{liuprod:eqn1}
	1+\sum_{n=1}^\infty (1+q^n)\frac{(q/c; q)_n (c/q)^n}{(c; q)_n}~ {_3 \phi_2} \biggl({{q^{-n}, q^n, q}
		\atop{ a, b}}; q, q\biggr)
	=\frac{(q, q, ac/q, bc/q; q)_\infty}{(a, b, c, c; q )_\infty}.
\end{align}
Using the $q$-Chu--Vandermonde summation (\ref{qchu-van}), we easily deduce that
\begin{equation}\label{liuprod:eqn2}
	{_2\phi_1}\biggl({{q^{-n}, q^n}
		\atop{-q}}; q, q\biggr)=\frac{\bigl(-q^{1-n}; q\bigr)_n q^{n^2}}{(-q; q)_n}
	=\frac{2q^{n(n+1)/2}}{1+q^n}.
\end{equation}	
Taking $a=c=-q$, $b=q$ in (\ref{liuprod:eqn1}) and then using the equation above, we obtain the following well-known identity for the sums of two squares (see, for example \cite[equation~(3.35)]{Andrews74}):
\begin{equation*}%\label{liuprod:eqn3}
\Biggl(\sum_{n=-\infty}^\infty (-1)^n q^{n^2}\Biggr)^{2}=1+4\sum_{n=1}^\infty (-1)^n\frac{q^{n(n+1)/2}}{1+q^n}.
\end{equation*}

Using the $q$-Pfaff--Saalsch\"utz summation formula (\ref{qpfaff}), it is found that
\begin{equation*}%\label{liuprod:eqn4}
	{_{3}\phi_{2}}\biggl({{q^{-n}, q^n, q}
		\atop{-q, -q}}; q, q\biggr)=\frac{(-1; q)^2_n q^n}{(-q; q)^2_n}=\frac{4q^n}{(1+q^n)^2}.
\end{equation*}	
Putting $a=b=c=-q$ in (\ref{liuprod:eqn1}) and then substituting the above equation into the resulting equation, we arrive at
the Jacobi identity for the sums of four squares (see, for example, \cite[equation~(3.40)]{Andrews74})
\begin{equation*}	%\label{liuprod:eqn5	}
	\Biggl(\sum_{n=-\infty}^\infty (-1)^n q^{n^2}\Biggr)^{4}=1+8\sum_{n=1}^\infty (-1)^n\frac{q^n}{(1+q^n)^2}.
\end{equation*}

Upon putting $b=0$, $a=c=-q$ in (\ref{liuprod:eqn1}), we immediately conclude that
\begin{equation*}	%\label{liuprod:eqn6}
\frac{(q; q)_\infty^3}{(-q; q)_\infty^3}
=1+2\sum_{n=1}^\infty (-1)^n ~ {_3 \phi_2} \biggl({{q^{-n}, q^n, q}
	\atop{ -q, 0}}; q, q\biggr).
\end{equation*}
By a direct computation, it is found that (\ref{spvalue:eqn5}) can be written as follows \cite[equation~(3.10)]{Liu2013IJTN}:
\begin{equation*}%\label{liuprod:eqn7}
	{_3\phi_2} \biggl({{q^{-n}, q^n, q}\atop{-q, 0}}; q, q\biggr)	
	=\frac{2q^n}{1+q^n}+(-1)^{n-1} \frac{(1-q^n)}{(1+q^n)}
	\sum_{|j|<n} (-1)^{j} q^{n^2-j^2}.
\end{equation*}
Combining the above two equations and using the Gauss identity (\ref{pre:eqn3}), we arrive at Andrews' identity for the sum of three squares \cite[equation~(5.16)]{Andrews86}
\begin{equation*}%\label{liuprod:eqn8}
	\Biggl(\sum_{n=-\infty}^\infty (-1)^n q^{n^2}\Biggr)^{3}
	=1+4\sum_{n=1}^\infty (-1)^n \frac{q^n}{1+q^n}
	-2\sum_{n=1}^\infty \frac{1-q^n}{1+q^n} \sum_{|j|<n}(-1)^j q^{n^2-j^2}.
\end{equation*}
It should be pointed out that based on this identity, D. Krammer \cite{Krammer1993} deduced the classical result of Legendre that a natural number is a sum of three squares if and only if it is not of the form $4^k (8l+7)$.

Letting $a=b=0$ in (\ref{liuprod:eqn1}) and then using (\ref{spvalue:eqn6}), we obtain
\begin{equation}\label{liuprod:eqn9}
\frac{(q; q)^2_\infty}{(c; q)^2_\infty}
=\sum_{n=0}^\infty (-1)^n \bigl(1-q^{2n+1}\bigr) \frac{(q/c; q)_n c^n}{(c; q)_{n+1}} \sum_{j=-n}^n (-1)^j q^{(3n^2+n)/2-(3j^2+j)/2}.
\end{equation}
When $c=0$, the above equation becomes \cite[equation~(1.2)]{Liu2002}
\begin{equation*} %\label{liuprod:eqn10}
(q; q)_\infty^2=\sum_{n=0}^\infty \sum_{j=-n}^n (-1)^j \bigl(1-q^{2n+1}\bigr) q^{2n^2+n-(3j^2+j)/2}.
\end{equation*}
Replacing $q$ by $q^2$ in (\ref{liuprod:eqn9}) and then setting $c=q$ using the Gauss identity (\ref{pre:eqn4}), we conclude~that%
\begin{equation*}%\label{liuprod:eqn11}
\Biggl(\sum_{n=0}^\infty q^{n(n+1)/2}\Biggr)^2
=\sum_{n=0}^\infty \sum_{j=-n}^n (-1)^{n+j} \bigl(1+q^{2n+1}\bigr) q^{3n^2+2n-3j^2-j}.
\end{equation*}

Putting $c=0$ in (\ref{liuprod:eqn1}) gives
\begin{equation*}%\label{liuprod:eqn13}
1+\sum_{n=1}^\infty (-1)^n q^{n(n-1)/2}~(1+q^n){_3 \phi_2} \biggl({{q^{-n}, q^n, q}
	\atop{ a, b}}; q, q\biggr)=\frac{(q; q)_\infty^2}{(a, b; q)_\infty}.
\end{equation*}
Replacing $q$ by $q^2$ in the equation above and then putting $a=-q$, $b=-q^2$ and finally using~(\ref{spvalue:eqn11}), we obtain \cite[equation~(1.7)]{Liu2002}
\begin{equation*}%\label{liuprod:eqn14}
(q; q)_\infty \bigl(q^2; q^2\bigr)_\infty=\sum_{n=0}^\infty \sum_{j=-n}^n \bigl(1-q^{2n+1}\bigr)
(-1)^j q^{2n^2+n-j^2}.
\end{equation*}

Setting $a=0$ and $b=-q$ in (\ref{liuprod:eqn1}) and then using (\ref{spvalue:eqn5}) in the resulting equation, we can find that
 for $c\not=q^{-m}$, $m=0, 1, 2, \ldots,$
\begin{equation}\label{exliuprod:eqn1}
	\frac{(q, q, -c; q)_\infty}{(c, c, -q; q)_\infty}
	=\sum_{n=0}^\infty \sum_{j=-n}^n (-1)^{n+j} \bigl(1-q^{2n+1}\bigr)
	\frac{(q/c; q)_n c^n}{(c; q)_{n+1}} q^{n^2-j^2}.
\end{equation}
Letting $c=0$ in (\ref{exliuprod:eqn1}), we immediately arrive at the following identity:
\begin{equation*}%\label{exliuprod:eqn2}
	\frac{(q; q)^2_\infty}{(-q; q)_\infty}
	=\sum_{n=0}^\infty \sum_{j=-n}^n (-1)^j \bigl(1-q^{2n+1}\bigr) q^{(3n^2+n)/2-j^2}.
\end{equation*}
Replacing $q$ by $q^2$ in (\ref{exliuprod:eqn1}) and then putting $c=-q$, we conclude that
\begin{equation*}%\label{expliuprod:eqn3}
	\frac{(q; q)_\infty \bigl(q^2; q^2\bigr)_\infty}{(-q; q)_\infty \bigl(-q; q^2\bigr)_\infty}
	=\sum_{n=0}^\infty \sum_{j=-n}^n (-1)^j \bigl(1-q^{2n+1}\bigr)
	q^{2n^2+n-2j^2}.
\end{equation*}

%%%%%%%%%%%%%%%%%%%%%%%%%%%%%%%%%%%%%%%%%%%%%%%%%%%%%%%%%%%%%%%%%%%%%%%%

Setting $\alpha=q$ and $\gamma=0$ in Theorem~\ref{prodliuthm}, we immediately arrive at
\begin{align} \label{liuprod:eqn15}
&\sum_{n=0}^\infty \frac{(q/c; q)_n (c/q)^n}{(qc; q)_n} \bigl(1-q^{2n+1}\bigr)
{_3\phi_2} \biggl({{q^{-n}, q^{n+1}, \beta ab}\atop{qa, qb}}; q, q\biggr)
=\frac{(q, ac, bc, \beta ab; q)_\infty}{(qa, qb, qc, \beta abc/q; q)_\infty}.
\end{align}

Replacing $q$ by $q^2$ in (\ref{liuprod:eqn15}) and then taking $\beta=q^2/{ab}$, we deduce that
\begin{align}
	&\sum_{n=0}^\infty \frac{\bigl(q^2/c; q^2\bigr)_n \bigl(c/q^2\bigr)^n}{\bigl(q^2c; q^2\bigr)_n} \bigl(1-q^{4n+2}\bigr)
	{_3\phi_2} \biggl({{q^{-2n}, q^{2n+2}, q^2}\atop{q^2a, q^2b}}; q^2, q^2\biggr)\nonumber\\
	&\qquad{}=\frac{\bigl(q^2, ac, bc, q^2; q^2\bigr)_\infty}{\bigl(q^2a, q^2b, qc^2, c; q^2\bigr)_\infty}.\label{liuprod:eqn20}
\end{align}
Taking $a=q$, $b=0$ in the equation above and then combining the resulting equation with~(\ref{spvalue:eqn7}), we find that
\begin{equation}\label{liuprod:eqn21}
\frac{\bigl(q^2, q^2, qc; q^2\bigr)_\infty}{\bigl(q, c, q^2c; q^2\bigr)_\infty}=\sum_{n=0}^\infty \sum_{j=0}^{2n} \bigl(1+q^{2n+1}\bigr) q^{2n^2+n-j(j+1)/2}
\frac{\bigl(q^2/c; q^2\bigr)_n c^n}{\bigl(q^2c; q^2\bigr)_n}.
\end{equation}
Letting $c=q$ in the equation above, we immediately arrive at
Andrews' identity for sums of three triangular numbers (\ref{trinumbers:eqn8}).

Taking $c=0$ and $c=-q$ respectively in (\ref{liuprod:eqn21}), we immediately deduce that
\begin{equation*}%\label{liuprod:eqn22}
\frac{\bigl(q^2; q^2\bigr)_\infty^2}{(q; q)_\infty}
=\sum_{n=0}^\infty \sum_{j=0}^{2n} (-1)^n \bigl(1+q^{2n+1}\bigr) q^{3n^2+2n-j(j+1)/2},
\end{equation*}
and
\begin{equation*}%\label{liuprod:eqn23}
\frac{(q; q)_\infty \bigl(q^4; q^4\bigr)^3_\infty}{\bigl(q^2; q^2\bigr)^2_\infty}
=\sum_{n=0}^\infty \sum_{j=0}^{2n} (-1)^n q^{3n^2+2n-j(j+1)/2}.	
\end{equation*}

Using the $q$-Pfaff--Saalsch\"utz summation formula (\ref{qpfaff}), it is found that
\begin{equation}\label{liuprod:eqn24}
{_3\phi_2} \biggl({{q^{-2n}, q^{2n+2}, q^2}\atop{q^3, q^3}}; q^2, q^2\biggr)
=\frac{(1-q)^2 q^{2n}}{\bigl(1-q^{2n+1}\bigr)^2}.
\end{equation}

Choosing $a=b=q$ in (\ref{liuprod:eqn20}) and then combining the resulting equation with (\ref{liuprod:eqn24}), we arrive at the following interesting $q$-identity:
\begin{equation}\label{liuprod:eqn25}
\frac{\bigl(q^2, q^2, qc, qc; q^2\bigr)_\infty}{\bigl(q, q, c, q^2c; q^2\bigr)_\infty}
=\sum_{n=0}^\infty \frac{\bigl(1+q^{2n+1}\bigr)}{\bigl(1-q^{2n+1}\bigr)} \frac{\bigl(q^2/c; q^2\bigr)_n c^n}{\bigl(q^2c; q^2\bigr)_n}.
\end{equation}

Putting $c=-q$ and $c=0$ respectively in the above equation, we find the following two generating functions for the sums of two triangular numbers:
\begin{equation} \label{liuprod:eqn26}
	\Biggl(\sum_{n=0}^\infty q^{n(n+1)}\Biggr)^2
	=\sum_{n=0}^\infty \frac{(-1)^n q^n}{1-q^{2n+1}}
\end{equation}
and
\begin{equation*}%\label{liuprod:eqn27}
	\Biggl(\sum_{n=0}^\infty q^{n(n+1)/2}\Biggr)^2
	=\sum_{n=0}^\infty (-1)^n \frac{\bigl(1+q^{2n+1}\bigr) q^{n^2+n}}{\bigl(1-q^{2n+1}\bigr)}.
\end{equation*}
Identity (\ref{liuprod:eqn26}) can be found in \cite[p.~397, Entry 18.2.4]{AndrewsBerndt2005}, and can also be found in
\cite[Theorem~3.8]{Liu2020HRjournal}.

Putting $c=q$ in (\ref{liuprod:eqn25}), we can arrive at the following identity for the sums of four triangular numbers, which is similar to Andrews' formula for the sums of three triangular numbers:
\begin{equation}\label{ftrinumbers}
\biggl(\sum_{n=0}^\infty q^{n(n+1)/2}\biggr)^4 =\sum_{n=0}^\infty \frac{\bigl(1+q^{2n+1}\bigr)q^n}{\bigl(1-q^{2n+1}\bigr)^2} .
\end{equation}
We can derive the Legendre formula for the sums of four triangular numbers from (\ref{ftrinumbers}) through some $q$-series manipulations \cite[Entry~18.2.5]{AndrewsBerndt2005}.

In their interesting paper, Chen and Wang \cite[equation~(3.43)]{ChenWang2020} prove that
\begin{equation}\label{cw:eqn1}
{_3\phi_2}\biggl({{q^{-n}, q^{n+1}, q}\atop{q^{3/2}, -q^{3/2}}}; q, q\biggr)
=\biggl(\frac{1-q}{1-q^{2n+1}}\biggr) q^{(n^2+3n)/2}
\sum_{j=-n}^n q^{-j(j+1)/2}.
\end{equation}
Let us choose $a=q^{1/2}$, $b=-q^{1/2}$ and $\beta=-1$ in (\ref{liuprod:eqn15}) and use (\ref{cw:eqn1}). Then we deduce that for~${|c|<1}$,
\begin{equation*}%\label{cw:eqn2}
\frac{(q; q)_\infty^2 \bigl(qc^2; q^2\bigr)_\infty}{(c, qc; q)_\infty \bigl(q; q^2\bigr)_\infty}
=\sum_{n=0}^\infty \sum_{j=-n}^n \frac{(q/c; q)_n c^n}{(qc; q)_{n}}
q^{n(n+1)/2-j(j+1)/2}.
\end{equation*}
When $c=0$, the above equation becomes
\begin{equation*}%\label{cw:eqn3}
	\frac{(q; q)_\infty^2 }{\bigl(q; q^2\bigr)_\infty}
	=\sum_{n=0}^\infty \sum_{j=-n}^n {(-1)^n} q^{n(n+1)-j(j+1)/2}.
\end{equation*}

\section{The proof of Theorem~\ref{trianHecke}}\label{sec5}
\subsection[An application of the Sears \_3phi\_2 transformation]{An application of the Sears $\boldsymbol{_3\phi_2}$ transformation}
The Sears $_3\phi_2$ transformation formula states that (see, for example, \cite[Theorem~3]{Liu2003})
\begin{align*}%\label{sears32trans:eqn1}
	{_3\phi_2}\biggl({{a_1, a_2, a_3}\atop{b_1, b_2}}; q, \frac{b_1b_2}{a_1a_2a_3}\biggr)
	=\frac{(b_2/a_3, b_1b_2/{a_1 a_2}; q)_\infty}
	{(b_2, b_1b_2/{a_1a_2a_3}; q)_\infty}
	{_3\phi_2}\biggl({{b_1/a_1, b_1/a_2, a_3}\atop{b_1, b_1b_2/{a_1a_2}}}; q, \frac{b_2}{a_3}\biggr).	
\end{align*}	
By making use of the above equation, we immediately conclude that
\begin{align*}
	&{_3\phi_2} \biggl({{q/c, \alpha ab/qu, \alpha ab/q}\atop{\alpha \beta ab/q, \alpha \gamma ab/q}}; q, \beta \gamma cu/q \biggr)\\
	&\qquad{}=\frac{\bigl(\alpha \gamma abc/q^2, \beta \gamma u; q\bigr)_\infty}
	{(\alpha \gamma ab/q, \beta \gamma cu/q; q)_\infty}
	{_3\phi_2} \biggl({{q/c, \beta u, \beta}\atop{\alpha \beta ab/q, \beta \gamma u}}; q, \frac{\alpha \gamma abc}{q^2}\biggr).
\end{align*}
Combining this equation with Corollary~\ref{liuthmc}, we are led to the following theorem.
\begin{Theorem}\label{liuthmd}
	For $\max\bigl\{\bigl|\alpha \beta abc/q^2\bigr|, \bigl|\alpha \gamma abc/q^2\bigr| \bigr\}<1$, we have
	\begin{align*}
		&\sum_{n=0}^\infty \frac{\bigl(1-\alpha q^{2n}\bigr)(\alpha, q/a, q/b, q/c; q)_n}{(1-\alpha)(q, \alpha a, \alpha b, \alpha c; q)_n}\biggl(\frac{\alpha abc}{q^2}\biggr)^n {_4\phi_3}\biggl({{q^{-n}, \alpha q^n, \beta, \gamma}
			\atop{q/a, q/b, \beta \gamma u}}; q, q\biggr)\\
		&\qquad{}=\frac{(q\alpha, \alpha ac/q, \alpha bc/q, \alpha \beta ab/q; q )_\infty}{\bigl(\alpha a, \alpha b, \alpha c, \alpha \beta abc/q^2; q \bigr)_\infty} {_3\phi_2} \biggl({{q/c, \beta u, \beta}\atop{\alpha \beta ab/q, \beta \gamma u}}; q, \frac{\alpha \gamma abc}{q^2}\biggr).
	\end{align*}
\end{Theorem}
Theorem~\ref{liuthmd} includes the following well-known result \cite[equation~(3.8.9)]{Gas+Rah} and \cite[Theorem~1.8]{Liu2013IJTN} as a special case, from which one can derive the Rogers--Ramanujan identities.
\begin{Proposition}\label{WRR} For $|\alpha ac/q|<1,$ we have the $q$-transformation formula
	\begin{align*}
		&\frac{(q\alpha, \alpha ac/q; q )_\infty}{(\alpha a, \alpha c; q )_\infty} {_3\phi_2} \biggl({{q/a, q/c, \alpha bd/q}\atop{\alpha b, \alpha d}}; q, \frac{\alpha ac}{q}\biggr)\\
		&\qquad{}=\sum_{n=0}^\infty \frac{\bigl(1-\alpha q^{2n}\bigr)(\alpha, q/a, q/b, q/c, q/d; q)_n}{(1-\alpha)(q, \alpha a, \alpha b, \alpha c, \alpha d; q)_n}\biggl(\frac{-\alpha^2 abcd}{q^2}\biggr)^n q^{n(n-1)/2}.
	\end{align*}	
\end{Proposition}
\begin{proof}
	Letting $\beta=q/a$, $\gamma=q/b$ and $u=\alpha abd/{q^2}$ in Theorem~\ref{liuthmd}, we immediately deduce that
	\begin{align*}
		&\sum_{n=0}^\infty \frac{\bigl(1-\alpha q^{2n}\bigr)(\alpha, q/a, q/b, q/c; q)_n}{(1-\alpha)(q, \alpha a, \alpha b, \alpha c; q)_n}\biggl(\frac{\alpha abc}{q^2}\biggr)^n {_2\phi_1}\biggl({{q^{-n}, \alpha q^n}
			\atop{\alpha d}}; q, q\biggr)\\
		&\qquad{}=\frac{(q\alpha, \alpha ac/q; q )_\infty}{(\alpha a, \alpha c; q )_\infty} {_3\phi_2} \biggl({{q/a, q/c, \alpha bd/q}\atop{\alpha b, \alpha d}}; q, \frac{\alpha ac}{q}\biggr).
	\end{align*}
	Using the $q$-Chu--Vandermonde summation formula in (\ref{qchu-van}), we deduce that
	\[
	{_2\phi_1}\biggl({{q^{-n}, \alpha q^n}
		\atop{\alpha d}}; q, q\biggr)=\frac{(q/d; q)_n}{(\alpha d; q)_n}
	(-\alpha d)^n q^{n(n-1)/2}.
	\]
	Combining the above two equations completes the proof of Proposition~\ref{WRR}.
\end{proof}

\subsection{The proof of Theorem~\ref{trianHecke}}
\begin{proof}
Setting $\alpha=q$, $u=at$ and $\beta=q/a$ in Theorem~\ref{liuthmd}, we immediately deduce that
\begin{align}
	&\sum_{n=0}^\infty\bigl(1-q^{2n+1}\bigr) \frac{(q/a, q/b, q/c; q)_n}{(qa, qb, qc; q)_n} \biggl(\frac{abc}{q}\biggr)^n {_3\phi_2}\biggl({{q^{-n}, q^{n+1}, \gamma}\atop{q/b, q\gamma t}}; q, q\biggr)\nonumber\\
	&\qquad{}=\frac{(q, ac; q)_\infty}{(qa, qc; q)_\infty}
	{_3\phi_2}\biggl({{q/a, q/c, qt}\atop{qb, q\gamma t}}; q, \frac{abc\gamma}{q}\biggr).\label{trih:eqn2}	
\end{align}	
Replacing $q$ by $q^2$ and then putting $t=0$, $b=q^{-1}$, $\gamma=q^2$ in the resulting equation gives 		
	\begin{align*}
		&	\sum_{n=0}^\infty \bigl(1-q^{4n+2}\bigr) \frac{\bigl(q^2/a, q^2/c; q^2\bigr)_n}{\bigl(q^2a, q^2c; q^2\bigr)_n} \biggl(\frac{ac}{q^3}\biggr)^n
		 \biggl(\frac{1-q^{2n+1}}{1-q}\biggr){_3\phi_2}\biggl({{q^{-2n}, q^{2n+2}, q^2}\atop{q^3, 0}}; q^2, q^2\biggr)\\
		&\qquad{}=\frac{\bigl(q^2, ac; q^2\bigr)_\infty}{\bigl(q^2a, q^2c; q^2\bigr)_\infty}
		{_2\phi_1}\biggl({{q^2/a, q^2/c}\atop{q}}; q^2, \frac{ac}{q}\biggr).
	\end{align*}
	Substituting (\ref{spvalue:eqn7}) into the equation above and simplifying, we complete the proof of Theorem~\ref{trianHecke}.
\end{proof}

By making use of (\ref{trih:eqn2}), we can also prove the following theorem.
\begin{Theorem}\label{mortrianHecke} We have
\begin{align*}%\label{mtri:eqn1}
&\sum_{n=0}^\infty \sum_{j=-n}^n (-1)^j \bigl(1-q^{2n+1}\bigr)q^{j^2-n} \frac{\bigl(q^2/a, q^2/c; q^2\bigr)_n (ac)^n}{\bigl(q^2 a, q^2 c; q^2\bigr)_n}\\
&\qquad{}=\frac{\bigl(q^2, ca; q^2\bigr)_\infty}{\bigl(q^2 a, q^2 c; q^2\bigr)_\infty}
\sum_{n=0}^\infty \frac{\bigl(q^2/a, q^2/c; q^2\bigr)_n (ac/q)^n}{(-q; q)_{2n} \bigl(1+q^{2n+1}\bigr)}.\nonumber
\end{align*}	
\end{Theorem}
\begin{proof} Replacing $q$ by $q^2$ in (\ref{trih:eqn2}) and then putting $b=-q$, $t=1$ and $\gamma=-1$ in the resulting equation, it is found that
\begin{align*}
	&(1+q)\sum_{n=0}^\infty(1-q^{2n+1}) \frac{\bigl(q^2/a, q^2/c; q^2\bigr)_n}{\bigl(q^2a, q^2 c; q^2\bigr)_n} \biggl(-\frac{ac}{q}\biggr)^n {_3\phi_2}\biggl({{q^{-2n}, q^{2n+2}, -1}\atop{-q, -q^2}}; q^2, q^2\biggr)\\
	&\qquad{}=\frac{\bigl(q^2, ac; q^2\bigr)_\infty}{\bigl(q^2a, q^2c; q^2\bigr)_\infty}
	{_3\phi_2}\biggl({{q^2/a, q^2/c, q^2}\atop{-q^3, {-q^2}}}; q^2, \frac{ac}{q}\biggr).\nonumber
\end{align*}		
Substituting (\ref{newvalue:eqn5}) into the equation above, we complete the proof of Theorem~\ref{mortrianHecke}.
\end{proof}

Letting $a=c=0$ in Theorem~\ref{mortrianHecke}, we immediately arrive at the identity
\begin{equation*}%\label{mtri:eqn2}
	\sum_{n=0}^\infty \frac{q^{2n^2+n}}{\bigl(1+q^{2n+1}\bigr) (-q; q)_{2n}}
	=\frac{1}{\bigl(q^2; q^2\bigr)_\infty}\sum_{n=0}^\infty \sum_{j=-n}^n (-1)^j
	\bigl(1-q^{2n+1}\bigr)q^{2n^2+j^2+n}.
\end{equation*}

By taking $(a, c)=(1, 0)$ and $(a, c)=(-1, 0)$ respectively in Theorem~\ref{mortrianHecke}, we find that
\begin{equation*}%\label{mtri:eqn3}
	\sum_{n=0}^\infty \frac{(-1)^n \bigl(q^2; q^2\bigr)_n q^{n^2}}{\bigl(1+q^{2n+1}\bigr) (-q; q)_{2n}}
	=\sum_{n=0}^\infty \sum_{j=-n}^n (-1)^{n+j}
	\bigl(1-q^{2n+1}\bigr)q^{n^2+j^2}
\end{equation*}
and
\begin{equation*}%\label{mtri:eqn4}
	\sum_{n=0}^\infty \frac{ q^{n^2}}{\bigl(1+q^{2n+1}\bigr) \bigl(-q; q^2\bigr)_{n}}
	=\frac{\bigl(-q^2; q^2\bigr)_\infty}{\bigl(q^2; q^2\bigr)_\infty}\sum_{n=0}^\infty \sum_{j=-n}^n (-1)^{j}
	\bigl(1-q^{2n+1}\bigr)q^{n^2+j^2}.
	\end{equation*}

Setting $a=q$ and $c=0$ in Theorem~\ref{mortrianHecke} and simplifying, we deduce that
\begin{equation*}%\label{mtri:eqn5}
\sum_{n=0}^\infty (-1)^n \frac{\bigl(q; q^2\bigr)_n q^{n(n+1)}}{\bigl(1+q^{2n+1}\bigr)(-q; q)_{2n}}
=\frac{\bigl(q; q^2\bigr)_\infty}{\bigl(q^2; q^2\bigr)_\infty}
\sum_{n=0}^{\infty}\sum_{j=-n}^{n}(-1)^{n+j} q^{n^2+n+j^2}.
\end{equation*}

\section[Some applications of Corollary~\ref{liuthmc} to Rogers--Hecke type series]{Some applications of Corollary~\ref{liuthmc} \\ to Rogers--Hecke type series}\label{sec6}

In this section, we will discuss the application of Corollary~\ref{liuthmc} to Rogers--Hecke type series.
\subsection{Application of Corollary~\ref{liuthmc}. Part I}
The first main result of this section is the following theorem.
\begin{Theorem}\label{andrrthm1} For $|a|<1$ and $|c|<1$, we have
	\begin{align*} %\label{andram:eqn1}
		&\sum_{n=0}^\infty \frac{(q/c; q)_n c^n}{\bigl(q^2; q^2\bigr)_n (1+aq^n)} q^{n(n+1)/2}\\
		&\qquad{}=\frac{(qa, qc, -ac; q)_\infty}{(q, ac, -a; q)_{\infty}}
		\sum_{n=0}^\infty \sum_{j=-n}^n (-1)^j \bigl(1-q^{2n+1}\bigr) q^{n^2-j^2}
		\frac{(q/a, q/c; q)_n (ac)^n}{(qa, qc; q)_n}. \nonumber
	\end{align*}	
\end{Theorem}
\begin{proof}
Setting $u=0$, $\gamma=q/a$ in Corollary~\ref{liuthmc}, we immediately find that for $\max\{|\alpha \beta abc/q^2|,\allowbreak |\alpha bc/q^2|\}<1$,
\begin{align}
	&\sum_{n=0}^\infty \frac{\bigl(1-\alpha q^{2n}\bigr)(\alpha, q/a, q/b, q/c; q)_n}{(1-\alpha)(q, \alpha a, \alpha b, \alpha c; q)_n}\biggl(\frac{\alpha abc}{q^2}\biggr)^n
	{_3\phi_2}\biggl({{q^{-n}, \alpha q^n, \beta}
		\atop{0, q/b}}; q, q\biggr)\nonumber\\
	&\qquad{}=\frac{(q\alpha, \alpha \beta ab/q, \alpha ac/q; q )_\infty}{\bigl(\alpha a,\alpha c, \alpha \beta abc/q^2; q \bigr)_\infty}
	\sum_{n=0}^\infty \frac{(q/c, \alpha ab/q; q)_n}{(q, \alpha b, \alpha \beta ab/q; q)_n}\biggl(-\frac{\alpha \beta bc}{q}\biggr)^n q^{n(n-1)/2}.\label{andram:eqn2}
\end{align}
Setting $\alpha=\beta=q$ and $b=-1$ in (\ref{andram:eqn2}), we deduce that
	\begin{align*}
		&\sum_{n=0}^\infty \frac{\bigl(1-q^{2n+1}\bigr)(q/a, q/c; q)_n}{(qa, q c; q)_n}\biggl(-\frac{ac}{q}\biggr)^n
		{_3\phi_2}\biggl({{q^{-n}, q^{n+1}, q}
			\atop{0, -q}}; q, q\biggr)\\
		&\qquad{}=\frac{(q, -a, ac; q )_\infty}{(qa, qc, -ac; q )_\infty}
		\sum_{n=0}^\infty \frac{(q/c; q)_n c^n}{\bigl(q^2; q^2\bigr)_n (1+aq^n)} q^{n(n+1)/2}.\nonumber
	\end{align*}	
	Substituting (\ref{spvalue:eqn3}) into the left-hand side of the equation above, we complete the proof of Theorem~\ref{andrrthm1}.
\end{proof}

Putting $a=c=0$ in Theorem~\ref{andrrthm1}, we arrive at the following identity of Rogers--Hecke type:
\begin{equation*}%\label{andram:eqn3}
	\sum_{n=0}^\infty (-1)^n \frac{q^{n^2+n}}{\bigl(q^2; q^2\bigr)_n}
	=\frac{1}{(q; q)_\infty} \sum_{n=0}^\infty \sum_{j=-n}^n (-1)^j \bigl(1-q^{2n+1}\bigr) q^{2n^2+n-j^2}.
\end{equation*}

If we specialize Theorem~\ref{andrrthm1} to the case when $c=1$, we conclude that
\begin{equation}\label{andram:eqn4}
	\sum_{n=0}^\infty \frac{q^{n(n+1)/2}}{(-q; q)_n (1+aq^n)}
	=\sum_{n=0}^\infty \sum_{j=-n}^n (-1)^j \bigl(1-q^{2n+1}\bigr) q^{n^2-j^2} \frac{(q/a; q)_n a^n}{(a; q)_{n+1}}.
\end{equation}

Letting $a=0$ in the above equation, we immediately obtain the Andrews--Dyson--Hickerson identity \cite{ADH1988}
\begin{equation*}%\label{andram:eqn5}
	\sum_{n=0}^\infty \frac{q^{n(n+1)/2}}{(-q; q)_n }
	=\sum_{n=0}^\infty \sum_{j=-n}^n (-1)^{n+j} \bigl(1-q^{2n+1}\bigr) q^{(3n^2+n)/2-j^2}.
\end{equation*}

For any nonnegative integer $m$, taking $a=q^{m+1}$ in (\ref{andram:eqn4}) and noting that $(q^{-m}; q)_n=0$ for~${n>m}$,
we can get (\ref{finitemock}).

Replacing $q$ by $q^2$ in (\ref{andram:eqn4}) and then setting $a=q$ and $a=-q$ respectively in the resulting equation, we deduce that
\begin{equation}\label{andram:eqn6}
	\sum_{n=0}^\infty \frac{q^{n(n+1)}}{\bigl(-q^2; q^2\bigr)_n \bigl(1+q^{2n+1}\bigr)}
	=\sum_{n=0}^\infty \sum_{j=-n}^n (-1)^{j} \bigl(1+q^{2n+1}\bigr) q^{2n^2+n-2j^2}
\end{equation}
and
\begin{equation*}%\label{andram:eqn7}
	\sum_{n=0}^\infty \frac{q^{n(n+1)}}{\bigl(-q^2; q^2\bigr)_n \bigl(1-q^{2n+1}\bigr)}
	=\sum_{n=0}^\infty \sum_{j=-n}^n (-1)^{n+j} \bigl(1-q^{2n+1}\bigr) q^{2n^2+n-2j^2}.
\end{equation*}
It should be pointed out that (\ref{andram:eqn6}) is different from the following identity of \cite[equation~(2.8)]{WangChern2019} due to Wang and Chern:
\begin{equation*}%\label{WangChern:eqn1}
	\sum_{n=0}^\infty \frac{q^{n(n+1)}}{\bigl(-q^2; q^2\bigr)_n \bigl(1+q^{2n+1}\bigr)}
	=\sum_{n=0}^\infty \sum_{j=-n}^n (-1)^{n+j} \bigl(1-q^{2n+1}\bigr) q^{2n^2+n-j^2}.
\end{equation*}

If we specialize Theorem~\ref{andrrthm1} to the case when $c=-1$, we deduce that
\begin{align}
	&\sum_{n=0}^\infty (-1)^n \frac{q^{n(n+1)/2}}{(q; q)_n (1+aq^n)}\nonumber\\
	&\qquad{}=\frac{(-q, a, a; q)_\infty}{(q, -a, -a; q)_\infty)_\infty}\sum_{n=0}^\infty \sum_{j=-n}^n (-1)^{n+j} \bigl(1-q^{2n+1}\bigr) q^{n^2-j^2} \frac{(q/a; q)_n a^n}{(a; q)_{n+1}}.\label{andram:eqn8}
\end{align}

The special case when $a=0$ of the above equation gives the third identity in \cite[Proposition~1.11]{Liu2013IJTN}
\begin{equation*}%\label{andram:eqn9}
	\sum_{n=0}^\infty (-1)^n\frac{q^{n(n+1)/2}}{(q; q)_n}
	=\frac{(q; q)_\infty}{(-q; q)_\infty}\sum_{n=0}^\infty \sum_{j=-n}^n (-1)^{j} \bigl(1-q^{2n+1}\bigr) q^{(3n^2+n)/2-j^2}.
\end{equation*}

Putting $a=q$ in (\ref{andram:eqn8}) and noting that $(1; q)_n=\delta_{n0}$, we find that
\begin{equation*}%\label{newconstant:eqn2}
\sum_{n=0}^\infty (-1)^n \frac{q^{n(n+1)/2}}{(q; q)_n \bigl(1+q^{n+1}\bigr)}
=\frac{(q; q)_\infty}{(-q; q)_\infty}.
\end{equation*}

Replacing $q$ by $q^2$ in (\ref{andram:eqn8}) then setting $a=q$ and $a=-q$ respectively in the resulting equation yields
\begin{align*}%\label{andram:eqn11}
	\sum_{n=0}^\infty \frac{(-1)^n q^{n(n+1)}}{\bigl(q^2; q^2\bigr)_n \bigl(1+q^{2n+1}\bigr)}
	=\frac{\bigl(q, q, -q^2; q^2\bigr)_\infty}{\bigl(-q, -q, q^2; q^2\bigr)_\infty}\sum_{n=0}^\infty \sum_{j=-n}^n (-1)^{n+j} \bigl(1+q^{2n+1}\bigr) q^{2n^2+n-2j^2}
\end{align*}
and
\begin{align*}%\label{andram:eqn12}
\sum_{n=0}^\infty \frac{(-1)^n q^{n(n+1)}}{\bigl(q^2; q^2\bigr)_n \bigl(1-q^{2n+1}\bigr)}
=\frac{\bigl(-q, -q, -q^2; q^2\bigr)_\infty}{\bigl(q, q, q^2; q^2\bigr)_\infty}\sum_{n=0}^\infty \sum_{j=-n}^n (-1)^{j} \bigl(1-q^{2n+1}\bigr) q^{2n^2+n-2j^2}.
\end{align*}

Putting $(a, c)=(1, 0)$ in Theorem~\ref{andrrthm1}, we arrive at the following identity of Rogers--Hecke type:
\begin{align*}%\label{andram:eqn13}
	\sum_{n=0}^\infty \frac{(-1)^n q^{n^2+n}}{\bigl(q^2; q^2\bigr)_n (1+q^n)}
	=\frac{1}{2(-q; q)_\infty}
	\sum_{n=0}^\infty \sum_{j=-n}^n(-1)^{n+j} \bigl(1-q^{2n+1}\bigr) q^{(3n^2+n)/2-j^2}.
\end{align*}

Letting $\alpha \to 1$, $\beta=q$ and $a=c=0$ in (\ref{andram:eqn2}) and simplifying, we conclude that
\begin{align}\label{andram:eqn14}
&(q; q)_\infty \sum_{n=0}^{\infty} \frac{q^{n^2} b^n}{(q, b; q)_n}
=1+\sum_{n=1}^{\infty}(1+q^n)\frac{(q/b; q)_n b^n}{(b; q)_n} q^{n(n-1)}~{_3\phi_2}\biggl({{q^{-n}, q^n, q}
		\atop{0, q/b}}; q, q\biggr).
\end{align}
Letting $b\to \infty$ in the equation above and then using (\ref{spvalue:eqn6}) in the resulting equation, we find~that%
\begin{equation*}%\label{andram:eqn15}
\sum_{n=0}^{\infty} (-1)^n \frac{q^{n(n+1)/2}}{(q; q)_n}
=\frac{1}{(q; q)_\infty}\sum_{n=0}^{\infty}\sum_{j=-n}^{n} (-1)^j \bigl(1-q^{2n+1}\bigr) q^{2n^2+n-(3j^2+j)/2}.
\end{equation*}
Replacing $q$ by $q^2$ in (\ref{andram:eqn14}) and then using \cite[equation~(5.33)]{ChenWang2020} in the resulting equation,
we can arrive at (\ref{trinumbers:eqn2}).

\subsection{Application of Corollary~\ref{liuthmc}. Part II}
The second main result of this section is the following theorem.
\begin{Theorem}\label{andrrthm2} We have
\begin{align*}%\label{rr:eqn1}
&\sum_{n=0}^\infty \sum_{j=0}^{2n} \bigl(1-q^{4n+2}\bigr)\frac{\bigl(q^2/a, q^2/c; q^2\bigr)_n (ac)^n}{\bigl(q^2 a, q^2c; q^2\bigr)_n} q^{2n^2-j(j+1)/2}\\
&\qquad{}=\frac{(1-a)\bigl(q^2; q^2\bigr)_\infty}{\bigl(1-ac/q^2\bigr)\bigl(q^2 c; q^2\bigr)_\infty}
\sum_{n=0}^\infty \frac{\bigl(q^2/c, a/q; q^2\bigr)_n (-c)^n}{(q; q)_{2n} \bigl(a; q^2\bigr)_n}q^{n(n-1)}.\nonumber
\end{align*}
\end{Theorem}
\begin{proof} Setting $\alpha=q$ in (\ref{andram:eqn2}) and then replacing $q$ by $q^2$ and finally taking $(b, \beta)=\bigl(q^{-1}, q\bigr)$, we arrive at
\begin{align*}%\label{rr:eqn2}
&\sum_{n=0}^\infty \bigl(1-q^{4n+2}\bigr) \frac{\bigl(q^2/a, q^2/c; q^2\bigr)_n \bigl(ac/q^3\bigr)^n}{\bigl(q^2 a, q^2 c; q^2\bigr)_n}
 \biggl(\frac{1-q^{2n+1}}{1-q}\biggr)
{_3\phi_2}\biggl({{q^{-2n}, q^{2n+2}, q} \atop{0, q^3}}; q^2, q^2\biggr)\nonumber\\
&\qquad{}=\frac{(1-a)\bigl(q^2; q^2\bigr)_\infty}{\bigl(1-ac/q^2\bigr)\bigl(q^2 c; q^2\bigr)_\infty}
\sum_{n=0}^\infty \frac{\bigl(q^2/c, a/q; q^2\bigr)_n (-c)^n}{(q; q)_{2n} \bigl(a; q^2\bigr)_n} q^{n(n-1)}.
\end{align*}	
Substituting (\ref{spvalue:eqn7})	into the equation above and simplifying, we complete the proof of Theorem~\ref{andrrthm2}.
\end{proof}

Putting $a=c=0$ in Theorem~\ref{andrrthm2}, we immediately arrive at the following identities of Rogers--Hecke type:
\begin{equation*}%\label{rr:eqn3}
\sum_{n=0}^\infty \frac{q^{2n^2}}{(q; q)_{2n}}
=\frac{1}{\bigl(q^2; q^2\bigr)_\infty}	
\sum_{n=0}^\infty \sum_{j=0}^{2n} \bigl(1-q^{4n+2}\bigr) q^{4n^2+2n-j(j+1)/2}.
\end{equation*}
If we specialize Theorem~\ref{andrrthm2} to the case when $a=0$ and $c=1$, we deduce that
\begin{equation*}%\label{rr:eqn4}
\sum_{n=0}^\infty(-1)^n\frac{q^{n^2-n}}{\bigl(q; q^2\bigr)_n}
=\sum_{n=0}^\infty \sum_{j=0}^{2n} (-1)^n \bigl(1-q^{4n+2}\bigr) q^{3n^2+n-j(j+1)/2}.
\end{equation*}
If we specialize Theorem~\ref{andrrthm2} to the case when $a=0$ and $c=-1$, we deduce that
\begin{equation*}%\label{rr:eqn5}
	\sum_{n=0}^\infty \frac{\bigl(-q^2; q^2\bigr)_n q^{n^2-n}}{(q; q)_{2n}}
	=\frac{\bigl(-q^2; q^2\bigr)_\infty}{\bigl(q^2; q^2\bigr)_\infty}\sum_{n=0}^\infty \sum_{j=0}^{2n} \bigl(1-q^{4n+2}\bigr) q^{3n^2+n-j(j+1)/2}.
\end{equation*}
By letting $(a, c)=(0, q)$ in Theorem~\ref{andrrthm2} and simplifying, we conclude that
\begin{equation*}%\label{rr:eqn6}
	\sum_{n=0}^\infty (-1)^n \frac{q^{n^2}}{\bigl(q^2; q^2\bigr)_n}
	=\frac{\bigl(q; q^2\bigr)_\infty}{\bigl(q^2; q^2\bigr)_\infty}\sum_{n=0}^\infty \sum_{j=0}^{2n} (-1)^n \bigl(1+q^{2n+1}\bigr) q^{3n^2+2n-j(j+1)/2}.
\end{equation*}
Upon taking $(a, c)=(0, -q)$ in Theorem~\ref{andrrthm2} and simplifying, we obtain
\begin{equation*}%\label{rr:eqn7}
	\sum_{n=0}^\infty \frac{\bigl(-q; q^2\bigr)_n q^{n^2}}{\bigl(q^2; q^2\bigr)_n}
	=\frac{\bigl(-q; q^2\bigr)_\infty}{\bigl(q^2; q^2\bigr)_\infty}\sum_{n=0}^\infty \sum_{j=0}^{2n} \bigl(1-q^{2n+1}\bigr) q^{3n^2+2n-j(j+1)/2}.
\end{equation*}
Putting $(a, c)=(-q, 0)$ in Theorem~\ref{andrrthm2} and making a simple calculation yields
\begin{equation*}%\label{rr:eqn8}
\sum_{n=0}^\infty \frac{\bigl(-1; q^2\bigr)_n q^{2n^2}}{(q; q)_{2n} \bigl(-q; q^2\bigr)_n}
=\frac{1}{\bigl(q^2; q^2\bigr)_\infty}
\sum_{n=0}^\infty \sum_{j=0}^{2n} \bigl(1+q^{2n+1}\bigr) q^{3n^2+2n-j(j+1)/2}.
\end{equation*}
%%%%%%%%%%%%%%%%%%%%%%%%%%%%%%%%%%%%%%%%%%%%%%%%%%%%%%%%%%%%%%%%%%%%%%%
\subsection{Application of Corollary~\ref{liuthmc}. Part III}
The Sears $_4\phi_3$ transformation in Proposition~\ref{Searstrs} can be restated as follows:
\begin{align*}
	&{_4\phi_3} \biggl({{q^{-n}, \alpha q^n, \beta, \gamma} \atop {c, d, q\alpha \beta \gamma/cd}} ; q, q \biggr)
	=\frac{(q\alpha/c, cd/\beta \gamma; q)_n}{(c, q\alpha \beta \gamma; q)_n} \biggl(\frac{\beta \gamma}{d}\biggr)^n
	{_4\phi_3} \biggl({{q^{-n}, \alpha q^n, d/\beta, d/\gamma} \atop {d, dc/\beta \gamma, q\alpha/c}} ; q, q \biggr).
\end{align*}
Setting $\gamma=0$ in the above equation, we immediately deduce that
\begin{align}\label{searogers:eqn1}
	{_3\phi_2} \biggl({{q^{-n}, \alpha q^n, \beta} \atop {c, d}} ; q, q \biggr)
	=(-c)^n q^{n(n-1)/2} \frac{(q\alpha/c; q)_n}{(c; q)_n}
	{_3\phi_2} \biggl({{q^{-n}, \alpha q^n, d/\beta} \atop {d, q\alpha/c}} ; q, \frac{q\beta}{c} \biggr).
\end{align}
Setting $\gamma=q/{a}$ in Corollary~\ref{liuthmc} and simplifying, it is found that
\begin{align*}%\label{searogers:eqn2}
	&\sum_{n=0}^\infty \frac{\bigl(1-\alpha q^{2n}\bigr)(\alpha, q/a, q/b, q/c; q)_n}{(1-\alpha)(q, \alpha a, \alpha b, \alpha c; q)_n}\biggl(\frac{\alpha abc}{q^2}\biggr)^n
 {_3\phi_2}\biggl({{q^{-n}, \alpha q^n, \beta}
		\atop{q/b, q\beta u/a}}; q, q\biggr)\\
	&\qquad{}=\frac{(q\alpha, \alpha ac/q, \alpha \beta ab/q, \beta cu/a; q )_\infty}{\bigl(\alpha a, \alpha c, \alpha \beta abc/q^2, q\beta u/a; q \bigr)_\infty}{_3\phi_2} \biggl({{q/c, \alpha ab/qu, \alpha ab/q}\atop{\alpha \beta ab/q, \alpha b}}; q, \frac{\beta cu}{a} \biggr).\nonumber
\end{align*}
With the help of (\ref{searogers:eqn1}), we deduce that
\begin{align*}%\label{searogers:eqn3}
{_3\phi_2}\biggl({{q^{-n}, \alpha q^n, \beta}
	\atop{q/b, q\beta u/a}}; q, q\biggr)
=\frac{(\alpha b; q)_n}{(q/b; q)_n} \Bigl(-\frac{q}{b}\Bigr)^n q^{n(n-1)/2} {_3\phi_2}\biggl({{q^{-n}, \alpha q^n, qu/a}
	\atop{\alpha b, q\beta u/a}}; q, \beta b\biggr).
\end{align*}
Combining the above two equations, we are led to the following $q$-transformation formula:
\begin{align*}%\label{searogers:eqn4}
	&\sum_{n=0}^\infty \frac{\bigl(1-\alpha q^{2n}\bigr)(\alpha, q/a, q/c; q)_n}{(1-\alpha)(q, \alpha a, \alpha c; q)_n}\biggl(-\frac{\alpha ac}{q}\biggr)^n q^{n(n-1)/2} {_3\phi_2}\biggl({{q^{-n}, \alpha q^n, qu/a}\atop{\alpha b, q\beta u/a}}; q, \beta b\biggr)\\
	&\qquad=\frac{(q\alpha, \alpha ac/q, \alpha \beta ab/q, \beta cu/a; q )_\infty}{\bigl(\alpha a, \alpha c, \alpha \beta abc/q^2, q\beta u/a; q \bigr)_\infty}{_3\phi_2} \biggl({{q/c, \alpha ab/qu, \alpha ab/q}\atop{\alpha \beta ab/q, \alpha b}}; q, \frac{\beta cu}{a} \biggr).\nonumber
\end{align*}
Setting $\alpha=q$, $u=0$, $\beta=b=-1$ in the above equation and simplifying, we arrive at
\begin{align*}
	&\sum_{n=0}^\infty \bigl(1-q^{2n+1}\bigr)\frac{(q/a, q/c; q)_n}{(q a, q c; q)_n}(-ac)^n q^{n(n-1)/2}~ {_2\phi_1}\biggl({{q^{-n}, q^{n+1}}\atop{-q}}; q, 1\biggr)\\
	&=\frac{(1-a)(q; q )_\infty}{(1-ac/q)(qc; q )_\infty}
	\sum_{n=0}^\infty \frac{(q/c, -a; q)_n (-c)^n}{\bigl(q^2; q^2\bigr)_n (a; q)_n} q^{n(n-1)/2}.
\end{align*}
Substituting (\ref{newvalue:eqn4}), we are led to the following theorem.
\begin{Theorem}\label{andrrthm3} We have
\begin{align*} %\label{searogers:eqn5}
&\sum_{n=0}^\infty \sum_{j=-n}^n (-1)^{j}\bigl(1-q^{2n+1}\bigr) q^{j^2-n}
\frac{(q/a, q/c; q)_n}{(q a, q c; q)_n}(ac)^n \\
&\qquad{}=\frac{(1-a)(q; q )_\infty}{(1-ac/q)(qc; q )_\infty}
\sum_{n=0}^\infty \frac{(q/c, -a; q)_n (-c)^n}{\bigl(q^2; q^2\bigr)_n (a; q)_n} q^{n(n-1)/2}.\nonumber
\end{align*}
\end{Theorem}
If we specialize Theorem~\ref{andrrthm3} to the case when $a=0$, then we find that
\begin{align}
&\sum_{n=0}^\infty \sum_{j=-n}^n (-1)^{n+j}\bigl(1-q^{2n+1}\bigr) q^{j^2+n(n-1)/2}
\frac{(q/c; q)_n}{(q c; q)_n}c^n \nonumber\\
&\qquad{}=\frac{(q; q )_\infty}{(qc; q )_\infty}
\sum_{n=0}^\infty \frac{(q/c; q)_n (-c)^n}{\bigl(q^2; q^2\bigr)_n } q^{n(n-1)/2}.
\label{searogers:eqn6}
\end{align}
Putting $c=0$ in the equation above, we arrive at the identity \cite[equation~(4.12)]{Liu2013IJTN}
\begin{equation*}%\label{searogers:eqn7}
\sum_{n=0}^\infty \frac{q^{n^2}}{\bigl(q^2; q^2\bigr)_n}
=\frac{1}{(q; q)_\infty}
\sum_{n=0}^\infty \sum_{j=-n}^n (-1)^{j} \bigl(1-q^{2n+1}\bigr) q^{n^2+j^2}.
\end{equation*}
Letting $c=1$ and $c=-1$ in (\ref{searogers:eqn6}) respectively, we conclude that \cite[equation~(4.13)]{Liu2013IJTN}
\begin{equation*}%\label{searogers:eqn8}
\sum_{n=0}^\infty (-1)^{n} \frac{q^{n(n-1)/2}}{(-q; q)_n}
=\sum_{n=0}^\infty \sum_{j=-n}^n (-1)^{j} \bigl(1-q^{2n+1}\bigr) q^{j^2+n(n-1)/2}
\end{equation*}
and \cite[equation~(4.14)]{Liu2013IJTN}
\begin{equation*}%\label{searogers:eqn9}
	\sum_{n=0}^\infty \frac{q^{n(n-1)/2}}{(q; q)_n}
	=\frac{(-q; q)_\infty}{(q; q)_\infty}\sum_{n=0}^\infty \sum_{j=-n}^n (-1)^{j} \bigl(1-q^{2n+1}\bigr) q^{j^2+n(n-1)/2}.
\end{equation*}
Replacing $q$ by $q^2$ in (\ref{searogers:eqn6}) and then putting $c=q$ and $c=-q$ respectively in the resulting equation, we deduce that
\begin{equation*}%\label{searogers:eqn10}
\sum_{n=0}^\infty (-1)^n \frac{\bigl(q; q^2\bigr)_n }{\bigl(q^4; q^4\bigr)_n} q^{n^2}
=\frac{\bigl(q; q^2\bigr)_\infty}{\bigl(q^2; q^2\bigr)_\infty}
\sum_{n=0}^\infty \sum_{j=-n}^n (-1)^{n+j} \bigl(1+q^{2j+1}\bigr) q^{2j^2+n^2}
\end{equation*}
and
\begin{equation*}%\label{searogers:eqn11}
	\sum_{n=0}^\infty \frac{\bigl(-q; q^2\bigr)_n }{\bigl(q^4; q^4\bigr)_n} q^{n^2}
	=\frac{\bigl(-q; q^2\bigr)_\infty}{\bigl(q^2; q^2\bigr)_\infty}
	\sum_{n=0}^\infty \sum_{j=-n}^n (-1)^{j} \bigl(1-q^{2j+1}\bigr) q^{2j^2+n^2},
\end{equation*}

If we specialize Theorem~\ref{andrrthm3} to the case when $c=0$, we find that
\begin{align}\label{searogers:eqn12}
\sum_{n=0}^\infty \frac{(-a; q)_n q^{n^2}}{\bigl(q^2; q^2\bigr)_n (a; q)_n}
=\frac{1}{(q; q)_\infty} \sum_{n=0}^\infty \sum_{j=-n}^n (-1)^{n+j}\bigl(1-q^{2n+1}\bigr) q^{j^2+n(n-1)/2}
\frac{(q/a; q)_n }{(a; q)_{n+1}}a^n. 	
\end{align}
Letting $a=-1$ in the equation above and noting that $(1; q)_n=\delta_{n0}$, we deduce that
\begin{equation*}%\label{searogers:eqn13}
2(q; q)_\infty=\sum_{n=0}^\infty \sum_{j=-n}^n (-1)^{j}\bigl(1-q^{2n+1}\bigr) q^{j^2+n(n-1)/2}.	
\end{equation*}
Replacing $q$ by $q^2$ in (\ref{searogers:eqn12}) and then putting $a=q$ and $a=-q$ respectively, we obtain that
\begin{align*}%\label{searogers:eqn14}
\sum_{n=0}^\infty \frac{\bigl(-q; q^2\bigr)_n}{\bigl(q^4; q^4\bigr)_n \bigl(q; q^2\bigr)_n} q^{2n^2}
=\frac{1}{\bigl(q^2; q^2\bigr)_\infty}
\sum_{n=0}^\infty \sum_{j=-n}^n (-1)^{n+j}\bigl(1+q^{2n+1}\bigr) q^{2j^2+n^2}
\end{align*}
and
\begin{align*}%\label{searogers:eqn15}
	\sum_{n=0}^\infty \frac{\bigl(q; q^2\bigr)_n q^{2n^2}}{\bigl(q^4; q^4\bigr)_n \bigl(-q; q^2\bigr)_n}
	=\frac{1}{\bigl(q^2; q^2\bigr)_\infty}
	\sum_{n=0}^\infty \sum_{j=-n}^n (-1)^{j}\bigl(1-q^{2n+1}\bigr) q^{2j^2+n^2}.
\end{align*}
Multiplying both sides of (\ref{searogers:eqn12}) by $(1-a)$ and then letting $a=1$, we deduce that
\begin{equation*}%\label{searogers:eqn16}
\sum_{n=1}^\infty \frac{(-1; q)_n q^{n^2}}{(q; q)_{n-1} \bigl(q^2; q^2\bigr)_n}=\frac{1}{(q; q)_\infty} \sum_{n=0}^\infty \sum_{j=-n}^n (-1)^{n+j}\bigl(1-q^{2n+1}\bigr) q^{j^2+n(n-1)/2}.
\end{equation*}

\section{The proof of Theorem~\ref{liuheckethm}}\label{sec7}

For simplicity, we denote the finite theta function $S_n(q)$ by
\begin{equation}\label{newmock:eqn1}
S_n(q):=\sum_{j=-n}^n (-1)^j q^{-j^2}.
\end{equation}
In this section, we will prove the following theorem, which is more general than Theorem~\ref{liuheckethm}.
\begin{Theorem} \label{andmockthm}
Let $S_n(q)$ be defined by \eqref{newmock:eqn1}. Then for $\max \{|a|, |c/q|, |ac/q|\}<1 $, we have
\begin{align*}%\label{newmock:eqn2}
&\sum_{n=0}^\infty \frac{(q/c; q)_n}{(-a; q)_n} \biggl(-\frac{c}{q}\biggr)^n
=\frac{(-q, a, c, -ac/q; q)_\infty}{(q, -a, -c/q, ac/q; q)_\infty}\\
	&\qquad{} \times \Biggl(2+\sum_{n=1}^\infty (1+q^n)q^{n^2-2n}\frac{(q/a, q/c; q)_n (ac)^n}{(a, c; q)_n}(q^n S_n(q)-S_{n-1}(q)) \Biggr).\nonumber
\end{align*}
\end{Theorem}
\begin{proof}
	Setting $\alpha \to 1$, $\gamma=q/a$ and $u=at$ in Theorem~\ref{liuthmd}, we immediately find that for $\max\bigl\{\bigl|\beta abc/q^2\bigr|, \bigl|bc/q\bigr|\bigr\}<1,$
	\begin{align}
		&1+\sum_{n=1}^\infty (1+q^n) \frac{(q/a, q/b, q/c; q)_n}{(a, b, c; q)_n} \biggl(\frac{abc}{q^2}\biggr)^n~{_3\phi_2}\biggl({{q^{-n}, q^n, \beta}\atop{q/b, q\beta t}}; q, q\biggr)\nonumber\\
		&\qquad{}=\frac{(q, ac/q, bc/q, \beta ab/q; q)_\infty}{\bigl(a, b, c, \beta abc/q^2; q\bigr)_\infty}~
		{_3\phi_2}\biggl({{q/c, \beta at, \beta}\atop{\beta ab/q, q\beta t}}; q, \frac{bc}{q}\biggr).\label{newmock:eqn3}
	\end{align}	
 When $a=q$ the left-hand side becomes $1$, and at the same time, the above equation reduces to the $q$-Gauss summation formula.

	Setting $b=-1$, $t=0$ and $\beta=q$ in the equation above, we deduce that
	\begin{align*}%\label{newmock:eqn4}
\sum_{n=0}^\infty \frac{(q/c; q)_n}{(-a; q)_n} \biggl(-\frac{c}{q}\biggr)^n
		&{}=\frac{(-q, a, c, -ac/q; q)_\infty}{(q, -a, -c/q, ac/q; q)_\infty}\\
		&\quad{} \times \Biggl(2+\sum_{n=1}^\infty (1+q^n)^2 \frac{(q/a, q/c; q)_n}{(a, c; q)_n}
		\biggl(-\frac{ac}{q^2}\biggr)^n {_3\phi_2}\biggl({{q^{-n}, q^n, q}\atop{-q, 0}}; q, q\biggr)\Biggr).\nonumber
	\end{align*}	
	 Substituting (\ref{spvalue:eqn5}) into the right-hand side of the equation above, we complete the proof of Theorem~\ref{andmockthm}.
\end{proof}

Let $c=0$ in Theorem~\ref{andmockthm}, we can immediately get Theorem~\ref{liuheckethm}.

Letting $a=0$ in (\ref{liuhr:eqn1}), we immediately get the following identity:
\begin{align}\label{newmock:eqn5}
	\frac{(q; q)_\infty}{(-q; q)_\infty}
	\Biggl(\sum_{n=0}^\infty q^{n(n-1)/2}\Biggr)
	=\sum_{n=0}^\infty \sum_{j=-n}^n (-1)^{j}\bigl(1+q^n-q^{3n+1}-q^{4n+2}\bigr) q^{2n^2-j^2}.
\end{align}

Suppose that $m$ is a nonnegative integer. Taking $a=q^{m+1}$ in (\ref{liuhr:eqn2}) and simplifying, we get the following identity:
\begin{align*}%\label{newmock:eqn6}
	\frac{(q; q)_m}{(-q; q)_m}\sum_{n=0}^\infty \frac{q^{n(n-1)/2}}{\bigl(-q^{m+1}; q\bigr)_n}
	&{}=2+2\sum_{n=1}^\infty (-1)^n (1+q^n) q^{n^2}\frac{(q; q)^2_m}{(q; q)_{m+n} (q; q)_{m-n}}\\
	&\quad{} +\sum_{n=1}^\infty \frac{\bigl(1-q^{2n}\bigr) (q; q)^2_m}{(q; q)_{m+n} (q; q)_{m-n}} \sum_{|j|<n} (-1)^j q^{2n^2-n-j^2}. \nonumber
\end{align*}

If we specialize (\ref{newmock:eqn3}) by taking $\beta=q$, we immediately conclude that
\begin{align}
		&1+\sum_{n=1}^\infty (1+q^n) \frac{(q/a, q/b, q/c; q)_n}{(a, b, c; q)_n} \biggl(\frac{abc}{q^2}\biggr)^n {_3\phi_2}\biggl({{q^{-n}, q^n, q}\atop{q/b, q^2 t}}; q, q\biggr)\nonumber\\
		&\qquad{}=\frac{(q, ac/q, bc/q, ab; q)_\infty}{(a, b, c, abc/q; q)_\infty}~
		\sum_{n=0}^{\infty}\frac{(q/c, qat; q)_n}{\bigl(ab, q^2t; q\bigr)_n}\biggl(\frac{bc}{q}\biggr)^n.\label{newmock:eqn7}
	\end{align}	
Setting $a=c=0$ in the equation above and simplifying, we deduce that
\begin{align}
&\sum_{n=0}^{\infty}\frac{(-b)^n}{\bigl(q^2t; q\bigr)_n} q^{n(n-1)/2}\nonumber\\
&\qquad{}=\frac{(b; q)_\infty} {(q; q)_\infty}
\Biggl(1+\sum_{n=1}^{\infty} (1+q^n) \frac{(q/b; q)_n b^n}{(b; q)_n} q^{n(n-1)}~{_3\phi_2}\biggl({{q^{-n}, q^n, q}\atop{q/b, q^2 t}}; q, q\biggr)\Biggr).
\label{newmock:eqn8}
\end{align}
Setting $b=-1$ and $t=0$ in the equation above and then using (\ref{spvalue:eqn5}) in the resulting equation, we can arrive at (\ref{newmock:eqn5}) again.

Replacing $q$ by $q^2$ in (\ref{newmock:eqn8}) and then putting $b=-q$ and $t=-q^{-2}$ and finally using (\ref{spvalue:eqn11}) in the resulting equation, we deduce that
\begin{equation*}%\label{newmock:eqn9}
\sum_{n=0}^{\infty}\frac{q^{n^2}}{\bigl(-q^2; q^2\bigr)_n}=\frac{\bigl(-q; q^2\bigr)_\infty}{\bigl(q^2; q^2\bigr)_\infty}
\sum_{n=0}^\infty \sum_{j=-n}^{n} (-1)^j \bigl(1-q^{4n+2}\bigr) q^{3n^2+n-j^2}.
\end{equation*}
Replacing $q$ by $q^2$ in (\ref{newmock:eqn8}) and then putting $b=-1$ and $t=-q^{-3}$ and finally using (\ref{spvalue:eqn11}) in the resulting equation, we find that
\begin{equation*}%\label{newmock:eqn10}
\sum_{n=0}^{\infty}\frac{q^{n(n-1)}}{\bigl(-q; q^2\bigr)_n}=\frac{\bigl(-q^2; q^2\bigr)_\infty}{\bigl(q^2; q^2\bigr)_\infty}
\sum_{n=0}^\infty \sum_{j=-n}^{n} (-1)^j \bigl(1+q^{2n}-q^{4n+1}-q^{6n+3}\bigr) q^{3n^2-j^2}.
\end{equation*}
Multiplying both sides of (\ref{newmock:eqn7}) by $(1-a)$ and then letting $a \to 1$ and putting $c=0$, we conclude~that
\begin{align}
		&\sum_{n=1}^\infty \bigl(1-q^{2n}\bigr) \frac{(q/b; q)_n}{(b; q)_n} (-b)^n q^{n(n-3)/2}~{_3\phi_2}\biggl({{q^{-n}, q^n, q}\atop{q/b, q^2 t}}; q, q\biggr)\nonumber\\
		&\qquad{}=
		\sum_{n=0}^{\infty}\frac{(qt; q)_n}{\bigl(b, q^2t; q\bigr)_n}(-b)^n q^{n(n-1)/2}.\label{newmock:eqn11}
	\end{align}	
Setting $t=0$ and $b=-1$ in the above equation and then using (\ref{spvalue:eqn5}) in the resulting equation, we find that
\begin{equation*}%\label{newmock:eqn12}
2+\sum_{n=0}^{\infty}\frac{q^{n(n+1)/2}}{(-q; q)_n}
=\sum_{n=0}^{\infty}\sum_{j=-n}^{n}(-1)^{n+j} \bigl(1-q^{4n+2}\bigr) q^{n(3n-1)/2-j^2}.
\end{equation*}
Replacing $q$ by $q^2$ in (\ref{newmock:eqn11}) and then putting $b=-q$ and $t=-q^{-2}$ in the resulting equation and finally using
(\ref{spvalue:eqn11}), we conclude that
\begin{equation*}%\label{newmock:eqn13}
1+2\sum_{n=1}^{\infty} \frac{q^{n^2}}{\bigl(1+q^{2n}\bigr)\bigl(-q; q^2\bigr)_n}
=\sum_{n=0}^{\infty}\sum_{j=-n}^{n}(-1)^{n+j} \bigl(1-q^{4n+2}\bigr) q^{2n^2-j^2}.
\end{equation*}
Replacing $q$ by $q^2$ in (\ref{newmock:eqn11}) and then putting $b=-1$ and $t=-q^{-3}$ in the resulting equation and finally using
(\ref{spvalue:eqn11}), we deduce that
\begin{align*}%\label{newmock:eqn14}
&2q+(1+q)\sum_{n=0}^{\infty} \frac{q^{n(n+1)}}{\bigl(-q^2; q^2\bigr)_n \bigl(1+q^{2n+1}\bigr)}\\
&\qquad{}=\sum_{n=0}^{\infty}\sum_{j=-n}^{n} (-1)^{n+j} \bigl(q-q^{4n+1}+q^{2n}-q^{6n+4}\bigr) q^{2n^2-n-j^2}. \nonumber
\end{align*}

\section{More Rogers--Hecke type series}\label{sec8}
In this section, we will continue to discuss the application of Theorem~\ref{liuthmd} to Rogers--Hecke type series.
\subsection{More on Rogers--Hecke type series. Part I}
\begin{Theorem}\label{rogersheck:n1} For $|ac|<1$, the following double $q$-formulas holds:
\begin{align}
&\sum_{n=0}^\infty \bigl(1-q^{2n+1}\bigr) \frac{\bigl(q^2/a, q^2/c; q^2\bigr)_n}{\bigl(q^2 a, q^2c; q^2\bigr)_n} (ac)^n \sum_{j=-n}^n (-1)^j q^{n^2-j^2}\nonumber\\
&\qquad{}=\frac{\bigl(q^2, ac; q^2\bigr)_\infty}{\bigl(q^2 a, q^2 c; q^2\bigr)_\infty}	
\sum_{n=0}^\infty \frac{\bigl(q^2/a, q^2/c; q^2\bigr)_n}{\bigl(1+q^{2n+1}\bigr)\bigl(q^4; q^4\bigr)_n} (-ac)^n.\label{rht:eqn1}	
\end{align}
\end{Theorem}
\begin{proof}
Setting $\alpha=q$, $u=at$ and $\beta=q/a$ in Theorem~\ref{liuthmd}, we immediately deduce that
\begin{align}
	&\sum_{n=0}^\infty\bigl(1-q^{2n+1}\bigr) \frac{(q/a, q/b, q/c; q)_n}{(qa, qb, qc; q)_n} \biggl(\frac{abc}{q}\biggr)^n {_3\phi_2}\biggl({{q^{-n}, q^{n+1}, \gamma}\atop{q/b, q\gamma t}}; q, q\biggr)\nonumber\\
	&\qquad{}=\frac{(q, ac; q)_\infty}{(qa, qc; q)_\infty}
	{_3\phi_2}\biggl({{q/a, q/c, qt}\atop{qb, q\gamma t}}; q, \frac{abc\gamma}{q}\biggr).\label{rht:eqn2}	
\end{align}		
	
Replacing $q$ by $q^2$ in (\ref{rht:eqn2}) and then setting $\gamma=q$, $b=-q$, $t=-1/q$ in the resulting equation and finally dividing both sides by $(1+q)$, we get that
\begin{align}	
&\sum_{n=0}^\infty \bigl(1-q^{2n+1}\bigr)	\frac{\bigl(q^2/a, q^2/c; q^2\bigr)_n}{\bigl(q^2a, q^2 c; q^2\bigr)_n}\biggl(-\frac{ac}{q}\biggr)^n {_3\phi_2}\biggl({{q^{-2n}, q^{2n+2}, q}\atop{-q, -q^2}}; q^2, q^2\biggr)\nonumber\\
&\qquad{}=\frac{\bigl(q^2, ac; q^2\bigr)_\infty}{\bigl(q^2a, q^2c; q^2\bigr)_\infty}
\sum_{n=0}^\infty \frac{\bigl(q^2/a, q^2/c; q^2\bigr)_n}{\bigl(1+q^{2n+1}\bigr)\bigl(q^4; q^4\bigr)_n} (-ac)^n.\label{rht:eqn3}
\end{align}
Letting $\alpha \to 1$ in Theorem~\ref{spwwthm} and then replacing $q$ by $q^2$ and finally setting $c=-q$, $d=-q^2$, we deduce that \cite[equation~(6.15)]{Andrews2012}, \cite[equation~(4.5)]{Liu2013IJTN}
\begin{align}\label{rht:eqn4}	
	{_3\phi_2}\left({{q^{-2n}, q^{2n+2}, q }\atop{-q, -q^2}}; q^2, q^2\right)=(-1)^n q^{n^2+n}\sum_{j=-n}^n (-1)^j q^{-j^2}.
\end{align}
Substituting (\ref{rht:eqn4}) into the left-hand side of (\ref{rht:eqn3}), we arrive at (\ref{rht:eqn1}). We complete the proof of Theorem~\ref{rogersheck:n1}.
\end{proof}

Many identities of Rogers--Hecke type can be derived from this theorem; we present only a~few of them here.
		
Setting $c=0$ in (\ref{rht:eqn1}) and then putting $a=1$ and $a=-1$, respectively, in the resulting equation, we immediately find that
\begin{align*}%\label{rht:eqn5}	
&\sum_{n=0}^\infty \frac{q^{n^2+n}}{\bigl(1+q^{2n+1}\bigr)\bigl(-q^2; q^2\bigr)_n}
=\sum_{n=0}^\infty \sum_{j=-n}^n (-1)^{n+j} \bigl(1-q^{2n+1}\bigr) q^{2n^2+n-j^2}
\end{align*}
and
\begin{align*}%\label{rht:eqn6}	
&\sum_{n=0}^\infty \frac{(-1)^n q^{n^2+n}}{\bigl(1+q^{2n+1}\bigr)\bigl(q^2; q^2\bigr)_n}
=\frac{\bigl(-q^2; q^2\bigr)_\infty}{\bigl(q^2; q^2\bigr)_\infty}\sum_{n=0}^\infty \sum_{j=-n}^n (-1)^{j} \bigl(1-q^{2n+1}\bigr) q^{2n^2+n-j^2}. 	
\end{align*}

If we specialize Theorem~\ref{rogersheck:n1} to the case when $a=c=0$, then we conclude that
\begin{align*}%\label{rht:eqn7}	
	\sum_{n=0}^\infty \frac{(-1)^n q^{2n^2+2n}}{\bigl(1+q^{2n+1}\bigr) \bigl(q^4; q^4\bigr)_n}
	=\frac{1}{\bigl(q^2; q^2\bigr)_\infty}\sum_{n=0}^\infty \sum_{j=-n}^n (-1)^j
	\bigl(1-q^{2n+1}\bigr) q^{3n^2+2n-j^2}.
\end{align*}

Taking $c=0$ in (\ref{rht:eqn1}) and then putting $a=1/q$ and $a=-1/q$ respectively in the resulting equation, we conclude that
\begin{align*}%\label{rht:eqn8}	
\sum_{n=0}^\infty \frac{\bigl(q; q^2\bigr)_{n+1} q^{n^2}}{\bigl(1+q^{2n+1}\bigr)\bigl(q^4; q^4\bigr)_n}
=\frac{\bigl(q; q^2\bigr)_\infty}{\bigl(q^2; q^2\bigr)_\infty}
\sum_{n=0}^\infty \sum_{j=-n}^n (-1)^{n+j} \bigl(1-q^{2n+1}\bigr)^2 q^{2n^2-j^2}
\end{align*}
and
\begin{align*}%\label{rht:eqn9}	
\sum_{n=0}^\infty \frac{(-1)^n\bigl(-q; q^2\bigr)_n q^{n^2}}{\bigl(q^4; q^4\bigr)_n}
=\frac{\bigl(-q; q^2\bigr)_\infty}{\bigl(q^2; q^2\bigr)_\infty}
\sum_{n=0}^\infty \sum_{j=-n}^n (-1)^{j} \bigl(1-q^{4n+2}\bigr) q^{2n^2-j^2}.
\end{align*}
Setting $(a, c)=(1, q)$ in (\ref{rht:eqn1}) and then replacing $q$ by $-q$ in the resulting equation, we find that~\cite[equation~(2.10)]{WangChern2019}
\begin{align*}%\label{rht:eqn10}	
\sum_{n=0}^\infty \frac{\bigl(-q; q^2\bigr)_n q^n }{\bigl(1-q^{2n+1}\bigr) \bigl(-q^2; q^2\bigr)_n}=\sum_{n=0}^\infty \sum_{j=-n}^n q^{n^2+n-j^2}.
\end{align*}
Using the same method as the derivation of Theorem~\ref{rogersheck:n1}, we can prove the following theorem.
\begin{Theorem}%\label{rogersheck:n2}
For $|ac|<q$, we have
\begin{align}
	&\sum_{n=0}^\infty \bigl(1-q^{4n+2}\bigr) \frac{\bigl(q^2/a, q^2/c; q^2\bigr)_n}{\bigl(q^2 a, q^2c; q^2\bigr)_n} (ac)^n \sum_{j=-n}^n (-1)^j q^{n^2-n-j^2}\nonumber\\
	&\qquad{}=\frac{2\bigl(q^2, ac; q^2\bigr)_\infty}{\bigl(q^2 a, q^2 c; q^2\bigr)_\infty}	\sum_{n=0}^\infty \frac{\bigl(q^2/a, q^2/c; q^2\bigr)_n}{\bigl(1+q^{2n}\bigr)\bigl(-q, q^2; q^2\bigr)_n }\biggl(-\frac{ac}{q}\biggr)^n.\label{rht:eqn11}	
\end{align}
\end{Theorem}
Putting $a=c=0$ in (\ref{rht:eqn11}), we are led to the following identity of Rogers--Hecke type:
\begin{equation*}%\label{rht:eqn12}	
\sum_{n=0}^\infty \frac{(-1)^n q^{2n^2+n}}{\bigl(1+q^{2n}\bigr)\bigl(-q, q^2; q^2\bigr)_n}
=\frac{1}{\bigl(q^2; q^2\bigr)_\infty}
\sum_{n=0}^\infty \sum_{j=-n}^n (-1)^j q^{3n^2+n-j^2}.
\end{equation*}

Taking $a=1$ and $c=0$ in (\ref{rht:eqn11}), we arrive at the following identity, which is equivalent to \cite[equation~(2.6)]{WangChern2019}:
\begin{equation*}%\label{rht:eqn13}	
1+2\sum_{n=1}^\infty \frac{q^{n^2}}{\bigl(1+q^{2n}\bigr)\bigl(-q; q^2\bigr)_n}
=\sum_{n=0}^\infty \sum_{j=-n}^n (-1)^{n+j} \bigl(1-q^{4n+2}\bigr) q^{2n^2-j^2}.
\end{equation*}

Upon taking $a=q$ and $c=-q$ in (\ref{rht:eqn11}) and simplifying, we conclude that
\begin{equation*}%\label{rht:eqn14}	
2\sum_{n=0}^\infty \frac{\bigl(q; q^2\bigr)_n q^n}{\bigl(1+q^{2n}\bigr)\bigl(q^2; q^2\bigr)_n}
=\frac{\bigl(q^2; q^4\bigr)_\infty}{\bigl(q^4; q^4\bigr)_\infty}
\sum_{n=0}^\infty \sum_{j=-n}^n (-1)^{n+j} q^{n^2+n-j^2}.
\end{equation*}	

\subsection{More Rogers--Hecke type series. Part II}
 The following theorem is the corrected version of \cite[Theorem~4.10]{Liu2013IJTN}.
\begin{Theorem}\label{eulerthm} For $|c|<1$, we have
	\begin{align}
		&\sum_{n=0}^\infty \frac{(q/a, q/c; q)_n (ac/q)^n}{\bigl(q^2; q^2\bigr)_n}\nonumber\\
		&\qquad{}=\frac{(qa, qc; q)_\infty} {(q, ac; q)_\infty}
		\sum_{n=0}^\infty \sum_{j=-n}^n (-1)^j \bigl(1-q^{2n+1}\bigr) q^{j^2-n}
		\frac{(q/a, q/c; q)_n}{(qa, qc; q)_n}(ac)^n.\label{euler:eqn1}
	\end{align}	
\end{Theorem}
\begin{proof}
Setting $\alpha=q$, $u=at$ and $\beta=q/a$ in Theorem~\ref{liuthmd}, we immediately deduce that
\begin{align*}%\label{euler:eqn2}
	&\sum_{n=0}^\infty\bigl(1-q^{2n+1}\bigr) \frac{(q/a, q/b, q/c; q)_n}{(qa, qb, qc; q)_n} \biggl(\frac{abc}{q}\biggr)^n {_3\phi_2}\biggl({{q^{-n}, q^{n+1}, \gamma}\atop{q/b, q\gamma t}}; q, q\biggr)\\
	&\qquad{}=\frac{(q, ac; q)_\infty}{(qa, qc; q)_\infty}
	{_3\phi_2}\biggl({{q/a, q/c, qt}\atop{qb, q\gamma t}}; q, \frac{abc\gamma}{q}\biggr).\nonumber
\end{align*}	
Choosing $\gamma=-1$, $t=0$ and $b=-1$ in the equation above, we conclude that
\begin{align*}
	&\sum_{n=0}^\infty (1-q^{2n+1})\frac{(q/a, q/c; q)_n}{(qa, qc; q)_n}\biggl(-\frac{ ac}{q}\biggr)^n {_3\phi_2}\biggl({{q^{-n}, q^{n+1}, -1}\atop{0, -q}}; q, q\biggr)\nonumber\\
	&\qquad{}=\frac{(q, ac; q)_\infty}{(qa, qc; q)_\infty} {_2\phi_1} \biggl({{q/a, q/c}\atop{-q}}; q, \frac{ac}{q} \biggr).%\label{euler:eqn3}
\end{align*}
Substituting (\ref{newvalue:eqn3}) into the equation above, we complete the proof of Theorem~\ref{eulerthm}.
\end{proof}

When $a=c=0$, Theorem~\ref{eulerthm} immediately reduces to the following identity:
\begin{equation*}%\label{euler:eqn4}
	\sum_{n=0}^\infty \frac {q^{n^2}}{\bigl(q^2; q^2\bigr)_n}
	=\frac{1}{(q; q)_\infty}
	\sum_{n=0}^\infty \sum_{j=-n}^n (-1)^j \bigl(1-q^{2n+1}\bigr) q^{j^2+n^2}.
\end{equation*}	
Putting $(a, c)=(1, 0)$ and $(a, c)=(-1, 0)$ in Theorem~\ref{eulerthm} respectively, we conclude that
\begin{equation*}%\label{euler:eqn5}
	\sum_{n=0}^\infty (-1)^n \frac{q^{n(n-1)/2}}{(-q; q)_n}
	=\sum_{n=0}^\infty \sum_{j=-n}^n (-1)^{n+j} \bigl(1-q^{2n+1}\bigr) q^{j^2+n(n-1)/2}
\end{equation*}
and
\begin{equation*}%\label{euler:eqn6}
	\sum_{n=0}^\infty \frac{q^{n(n-1)/2}}{(-q; q)_n}
	=\frac{(-q; q)_\infty}{(q; q)_\infty}\sum_{n=0}^\infty \sum_{j=-n}^n (-1)^{j} \bigl(1-q^{2n+1}\bigr) q^{j^2+n(n-1)/2}.
\end{equation*}

\begin{Theorem}\label{liuoldthm} For $|ac/q|<1$, we have
	\begin{align*}%\label{liu2013:eqn1}
		&\frac{\bigl(q^2, ac; q^2\bigr)_\infty}{\bigl(q^2 a, q^2 c; q^2\bigr)_\infty}
		\sum_{n=0}^\infty \frac{\bigl(q^2/a, q^2/c; q^2\bigr)_n (ac/q)^n}{(q; q)_{2n} (1+q^{2n+1})}\\
		&\qquad{}=\sum_{n=0}^\infty \sum_{j=-n}^n \bigl(1-q^{2n+1}\bigr) q^{n^2-j^2-j}
		\frac{\bigl(q^2/a, q^2/c; q^2\bigr)_n (ac)^n}{\bigl(q^2 a, q^2 c; q^2\bigr)_n}.\nonumber
	\end{align*}
\end{Theorem}
\begin{proof} Setting $b=q^{-1/2}$, $\gamma=q$ and $t=-q^{-1/2}$ in (\ref{trih:eqn2}), we can deduce that
	\begin{align*}
		&\sum_{n=0}^\infty \bigl(1-q^{2n+1}\bigr)\frac{(q/a, q/c)_n (ac)^n}{(qa, qc; q)_n}
		\frac{\bigl(1-q^{n+1/2}\bigr)}{\bigl(1-q^{1/2}\bigr)} q^{-3n/2}{_3\phi_2}\biggl({{q^{-n}, q^{n+1}, q}\atop{q^{3/2}, -q^{3/2}}}; q, q\biggr)\\
		&\qquad =\frac{(q, ac; q)_\infty}{(qa, qc; q)_\infty}
		\sum_{n=0}^\infty \frac{\bigl(q/a, q/c, -q^{1/2}; q\bigr)_n}{\bigl(q, q^{1/2}, -q^{3/2}; q\bigr)_n}\biggl(\frac{ac}{q^{1/2}}\biggr)^n.
	\end{align*}
	Substituting (\ref{cw:eqn1}) into the left-hand side of the equation above and then replacing $q$ by $q^2$ and simplifying, we complete the proof of Theorem~\ref{liuoldthm}.
\end{proof}

If we specialize Theorem~\ref{liuoldthm} to the case when $c=0$, then we find that
\begin{align}
	&\frac{\bigl(q^2; q^2\bigr)_\infty}{\bigl(q^2a; q^2\bigr)_\infty}\sum_{n=0}^\infty
	\frac{\bigl(q^2/a; q^2\bigr)_n (-a)^n }{(q; q)_{2n} \bigl(1+q^{2n+1}\bigr)}q^{n^2} \nonumber\\
	&\qquad{}=\sum_{n=0}^\infty \sum_{j=-n}^n (-1)^n\bigl(1-q^{2n+1}\bigr) q^{2n^2+n-j^2-j}
	\frac{\bigl(q^2/a; q^2\bigr)_n a^n}{\bigl(q^2a; q^2\bigr)_n}.\label{liu2013:eqn2}
\end{align}
Setting $a=1$ in the equation above, we immediately arrive at the following beautiful identity of Rogers--Hecke type due to Wang and Chern \cite[equation~(2.9)]{WangChern2019}:
\begin{equation}\label{liu2013:eqn3}
	\sum_{n=0}^\infty \frac{(-1)^n q^{n^2}}{\bigl(q; q^2\bigr)_n \bigl(1+q^{2n+1}\bigr)}
	=\sum_{n=0}^\infty \sum_{j=-n}^n (-1)^n \bigl(1-q^{2n+1}\bigr) q^{2n^2+n-j^2-j}.
\end{equation}
When $a=-1$, (\ref{liu2013:eqn2}) becomes
\begin{equation*}%\label{liu2013:eqn4}
	\sum_{n=0}^\infty
	\frac{\bigl(-q^2; q^2\bigr)_n q^{n^2}}{(q; q)_{2n} \bigl(1+q^{2n+1}\bigr)}
	=\frac{\bigl(-q^2; q^2\bigr)_\infty}{\bigl(q^2; q^2\bigr)_\infty}
	\sum_{n=0}^\infty \sum_{j=-n}^n \bigl(1-q^{2n+1}\bigr) q^{2n^2+n-j^2-j}.
\end{equation*}
Taking $a=q$ in (\ref{liu2013:eqn2}), we find that
\begin{equation*}%\label{liu2013:eqn5}
	\sum_{n=0}^\infty (-1)^n \frac{q^{n^2+n}}{\bigl(1+q^{2n+1}\bigr)\bigl(q^2; q^2\bigr)_n}
	=\frac{\bigl(q; q^2\bigr)_\infty}{\bigl(q^2; q^2\bigr)_\infty}
	\sum_{n=0}^\infty \sum_{j=-n}^n q^{2n^2+2n-j^2-j}.
\end{equation*}

Letting $a=c=q$ in Theorem~\ref{liuoldthm} and simplifying, we conclude that
\begin{equation*}%\label{liu2013:eqn6}
	\sum_{n=0}^\infty\frac{\bigl(q; q^2\bigr)_n q^n}{\bigl(q^2; q^2\bigr)_n \bigl(1+q^{2n+1}\bigr)}
	=\frac{\bigl(q; q^2\bigr)_\infty^2}{\bigl(q^2; q^2\bigr)^2_\infty}
	\sum_{n=0}^\infty \sum_{j=-n}^n \frac{q^{n^2+2n-j^2-j}}{1-q^{2n+1}}.
\end{equation*}

If we specialize Theorem~\ref{liuoldthm} to the case when $a=1$, then we find that
\begin{align*} %\label{liu2013:eqn7}
	\sum_{n=0}^\infty \frac{\bigl(q^2/c; q^2\bigr)_n (c/q)^n}{\bigl(q; q^2\bigr)_n \bigl(1+q^{2n+1}\bigr)}
	=\sum_{n=0}^\infty \sum_{j=-n}^n \bigl(1-q^{2n+1}\bigr) q^{n^2-j^2-j} \frac{\bigl(q^2/c; q^2\bigr)_n c^n}{\bigl(c; q^2\bigr)_{n+1}}.
\end{align*}
Equating the coefficients of $c$ on both sides of the above equation and then combining the resulting equation with (\ref{liu2013:eqn3}) and finally replacing $q$ by $-q$, we deduce that
\begin{align*}
\sum_{n=0}^{\infty} \frac{ q^{n^2-2n}}{\bigl(-q; q^2\bigr)_n \bigl(1-q^{2n+1}\bigr)}
=\sum_{n=0}^{\infty} \sum_{j=-n}^{n} q^{2n^2-n-j^2-j}\bigl(1+q^{2n+1}\bigr)\bigl(1-q^{2n}+q^{4n+2}\bigr).
\end{align*}
%%%%%%%%%%%%%%%%%%%%%%%%%%%%%%%%%%%%%%%%%%%%%%%%%%%%%%%%%%%%%%%%%%%
\subsection{More Rogers--Hecke type series. Part III}
Replacing $q$ by $q^2$ in (\ref{euler:eqn1}), we have
\begin{align*}
	&\sum_{n=0}^\infty\bigl(1-q^{4n+2}\bigr) \frac{\bigl(q^2/a, q^2/b, q^2/c; q^2\bigr)_n}{\bigl(q^2a, q^2b, q^2c; q^2\bigr)_n} \biggl(\frac{abc}{q^2}\biggr)^n {_3\phi_2}\biggl({{q^{-2n}, q^{2n+2}, \gamma}\atop{q^2/b, q^2\gamma t}}; q^2, q^2\biggr)\nonumber\\
	&\qquad{}=\frac{\bigl(q^2, ac; q^2\bigr)_\infty}{\bigl(q^2a, q^2c; q\bigr)_\infty}
	{_3\phi_2}\biggl({{q^2/a, q^2/c, q^2t}\atop{q^2b, q^2\gamma t}}; q^2, \frac{abc\gamma}{q^2}\biggr).%\label{euler:eqn7}
\end{align*}	
Choosing $b=-1$, $\gamma=-1$ and $t=1/q$ in the equation above, we conclude that
\begin{align*}%\label{euler:eqn8}
& \sum_{n=0}^\infty(1-q^{4n+2}) \frac{\bigl(q^2/a, q^2/c; q^2\bigr)_n}{\bigl(q^2a, q^2c; q^2\bigr)_n} \biggl(-\frac{ac}{q^2}\biggr)^n {_3\phi_2}\biggl({{q^{-2n}, q^{2n+2}, -1}\atop{-q, -q^2}}; q^2, q^2\biggr)\\
&\qquad{}	=\frac{\bigl(q^2, ac; q^2\bigr)_\infty}{\bigl(q^2a, q^2c; q\bigr)_\infty}
	{_3\phi_2}\biggl({{q^2/a, q^2/c, q}\atop{-q^2, -q}}; q^2, \frac{ac}{q^2}\biggr).\nonumber
\end{align*}	
Substituting (\ref{newvalue:eqn3}) into the left-hand side of the equation above, we get the following theorem.
\begin{Theorem} \label{gaussthm}We have
\begin{align*}%\label{euler:eqn9}
&\sum_{n=0}^\infty \sum_{j=-n}^n (-1)^j \bigl(1-q^{4n+2}\bigr)	q^{j^2-2n}
\frac{\bigl(q^2/a, q^2/c; q^2\bigr)_n (ac)^n}{\bigl(q^2a, q^2c; q^2\bigr)_n}\\
&\qquad{}=\frac{\bigl(q^2, ac; q^2\bigr)_\infty}{\bigl(q^2a, q^2c; q\bigr)_\infty}
\sum_{n=0}^\infty \frac{\bigl(q^2/a, q^2/c, q; q^2\bigr)_n}{(-q; q)_{2n} \bigl(q^2; q^2\bigr)_n}\biggl(\frac{ac}{q^2}\biggr)^n.
\nonumber
\end{align*}	
\end{Theorem}
Letting $a=c=0$ in Theorem~\ref{gaussthm} and simplifying, we conclude that
\begin{equation*}%\label{euler:eqn10}
\sum_{n=0}^\infty \frac{\bigl(q; q^2\bigr)_n }{\bigl(q^2; q^2\bigr)_n (-q; q)_{2n}}q^{2n^2}
=\frac{1}{\bigl(q^2; q^2\bigr)_\infty}
\sum_{n=0}^\infty \sum_{j=-n}^n (-1)^j \bigl(1-q^{4n+2}\bigr) q^{j^2+2n^2}.
\end{equation*}
 If we specialize Theorem~\ref{gaussthm} to the case when $a=1$, then we find that
 \begin{align}\label{euler:eqn11}
 	\sum_{n=0}^\infty \sum_{j=-n}^n (-1)^j \bigl(1-q^{4n+2}\bigr)	q^{j^2-2n}
 	\frac{\bigl(q^2/c; q^2\bigr)_n c^n}{\bigl(c; q^2\bigr)_{n+1}}
 	=\sum_{n=0}^\infty \frac{\bigl(q^2/c, q; q^2\bigr)_n}{(-q; q)_{2n}}\biggl(\frac{c}{q^2}\biggr)^n.
 \end{align}
Letting $c=0$ in the equation above, we immediately find that
\begin{equation}\label{euler:eqn12}
\sum_{n=0}^\infty (-1)^n \frac{\bigl(q; q^2\bigr)_n}{(-q; q)_{2n}} q^{n^2-n}
= \sum_{n=0}^\infty \sum_{j=-n}^n (-1)^{j+n} \bigl(1-q^{4n+2}\bigr)	q^{j^2+n^2-n}.
\end{equation}	
Applying the $q$-derivative operator $\partial_{q^2, c}$ to both sides of $(\ref{euler:eqn11})$ and simplifying, we conclude that
\begin{align}\label{euler:eqn13}
 &\sum_{n=0}^\infty \frac{\bigl(1-q^{2n}\bigr)\bigl(q; q^2\bigr)_n\bigl(q^2/c; q^2\bigr)_{n-1}c^{n-1}q^{-2n}}{(-q; q)_{2n}}\\
 &\quad{}=\sum_{n=0}^\infty \sum_{j=-n}^n (-1)^j \bigl(1-q^{4n+2}\bigr) \bigl(1+(c-2)q^{2n}-cq^{2n+2}+q^{4n+2}\bigr)	q^{j^2-2n}
 	\frac{\bigl(q^2/c; q^2\bigr)_{n-1} c^{n-1}}{\bigl(c; q^2\bigr)_{n+2}}. \nonumber
 \end{align}
This equation can also be used to derive identities for Rogers--Hecke type series. Putting $c=0$ in (\ref{euler:eqn13}), we find that
\begin{align}
&\sum_{n=0}^{\infty} (-1)^n \frac{\bigl(1-q^{2n}\bigr)\bigl(q; q^2\bigr)_n }{(-q; q)_{2n}} q^{n^2-3n}\nonumber\\
&\qquad{}=\sum_{n=0}^\infty \sum_{j=-n}^n (-1)^{n+j} \bigl(1-q^{4n+2}\bigr) \bigl(1-2q^{2n}+q^{4n+2}\bigr)	q^{n^2+j^2-3n}.\label{euler:eqn14}
\end{align}
By adding (\ref{euler:eqn12}) and (\ref{euler:eqn14}) together, we conclude that
\begin{equation*}%\label{euler:eqn15}
\sum_{n=0}^{\infty} (-1)^n \frac{\bigl(q; q^2\bigr)_n }{(-q; q)_{2n}} q^{n^2-3n}
=\sum_{n=0}^\infty \sum_{j=-n}^n (-1)^{n+j} \bigl(1-q^{4n+2}\bigr) \bigl(1-q^{2n}+q^{4n+2}\bigr)	q^{n^2+j^2-3n}.
\end{equation*}

\begin{Theorem}\label{Gaussthm} We have
\begin{align*} %\label{Gauus；eqn0}
&\frac{\bigl(q^2, ac, c/q, qa; q^2\bigr)_\infty}{\bigl(q^2a, q, q^2c, ac/q; q^2\bigr)_\infty}
\sum_{n=0}^\infty \frac{\bigl(q^2/c; q^2\bigr)_n (c/q)^n}{\bigl(qa; q^2\bigr)_n}\\
&\qquad{}=\sum_{n=0}^\infty \sum_{j=0}^{2n} \bigl(1-q^{4n+2}\bigr) q^{2n^2-j(j+1)/2}
\frac{\bigl(q^2/a, q^2/c; q^2\bigr)_n}{\bigl(q^2a, q^2c; q^2\bigr)_n} (ac)^n.\nonumber
\end{align*}	
\end{Theorem}
\begin{proof}
Taking $\alpha=q$, $\gamma=q/a$, $u=at$ in Theorem~\ref{liuthmd} and replacing $q$ by $q^2$, we deduce that
 \begin{align*}%\label{Gauss:eqn1}
 	&\sum_{n=0}^\infty\bigl(1-q^{4n+2}\bigr) \frac{\bigl(q^2/a, q^2/b, q^2/c; q\bigr)_n}{\bigl(q^2a, q^2b, q^2c; q^2\bigr)_n} \biggl(\frac{abc}{q^2}\biggr)^n {_3\phi_2}\biggl({{q^{-2n}, q^{2n+2}, \beta}\atop{q^2/b, q^2\beta t}}; q^2, q^2\biggr)\\
&\qquad{} 	=\frac{\bigl(q^2, ac, bc, \beta ab; q^2\bigr)_\infty}{\bigl(q^2a, q^2b, q^2c, \beta abc/q^2; q^2\bigr)_\infty}
 	{_3\phi_2}\biggl({{q^2/c, \beta, \beta at}\atop{\beta ab, q^2\beta t}}; q^2, bc\biggr).\nonumber
 \end{align*}	
 Making the change $(b, \beta, t)$ to $\bigl(q^{-1}, q^2, 0\bigr)$ in the equation above, we obtain
\begin{align*}%\label{Gauss:eqn2}
	&\sum_{n=0}^\infty\bigl(1-q^{4n+2}\bigr) \frac{\bigl(q^2/a, q^3, q^2/c; q\bigr)_n}{\bigl(q^2a, q, q^2c; q^2\bigr)_n} \biggl(\frac{ac}{q^3}\biggr)^n {_3\phi_2}\biggl({{q^{-2n}, q^{2n+2}, q^2}\atop{q^3, 0}}; q^2, q^2\biggr)\\
&\qquad{}	=\frac{\bigl(q^2, ac, c/q, qa; q^2\bigr)_\infty}{\bigl(q^2a, q, q^2c, ac/q; q\bigr)_\infty}
	{_2\phi_1}\biggl({{q^2/c, q^2}\atop{qa}}; q^2, \frac{c}{q}\biggr). \nonumber
\end{align*}	
Substituting (\ref{spvalue:eqn7}) into the equation above and simplifying, we complete the proof of Theorem~\ref{Gaussthm}.
\end{proof}

If we let $c=0$ in Theorem~\ref{Gaussthm}, then we are led to the $q$-identity{\samepage
\begin{align}
&\frac{\bigl(q^2, qa; q^2\bigr)_\infty}{\bigl(q, q^2a; q^2\bigr)_\infty}
\sum_{n=0}^\infty (-1)^n\frac{q^{n^2}}{\bigl(qa; q^2\bigr)_n}\nonumber\\
&\qquad{}=\sum_{n=0}^\infty \sum_{j=0}^{2n} (-1)^n \bigl(1-q^{4n+2}\bigr) q^{3n^2+n-j(j+1)/2}
\frac{\bigl(q^2/a; q^2\bigr)_n a^n}{\bigl(q^2a; q^2\bigr)_n}.\label{Gauss:eqn3}
\end{align}
When $a=1,$ the above equation reduces to (\ref{trinumbers:eqn3}).}

Letting $a=q$ and $a=-q$ respectively in (\ref{Gauss:eqn3}), we conclude that
\begin{equation*}%\label{Gauss:eqn4}
\sum_{n=0}^\infty (-1)^n\frac{q^{n^2}}{\bigl(q^2; q^2\bigr)_n}
=\frac{\bigl(q; q^2\bigr)_\infty^2}{\bigl(q^2; q^2\bigr)^2_\infty}
\sum_{n=0}^\infty \sum_{j=0}^{2n} (-1)^n \bigl(1+q^{2n+1}\bigr) q^{3n^2+2n-j(j+1)/2}
\end{equation*}
and
\begin{equation*}%\label{Gauss:eqn5}
	\sum_{n=0}^\infty (-1)^n\frac{q^{n^2}}{\bigl(-q^2; q^2\bigr)_n}
	=\frac{\bigl(q^2; q^4\bigr)_\infty }{\bigl(q^4; q^4\bigr)_\infty}
	\sum_{n=0}^\infty \sum_{j=0}^{2n} \bigl(1-q^{2n+1}\bigr) q^{3n^2+2n-j(j+1)/2}.
\end{equation*}

Putting $a=-1$ and $a=0$ respectively in (\ref{Gauss:eqn3}), we are led to the following two identities:
\begin{align*}%\label{Gauss:eqn7}
\sum_{n=0}^\infty (-1)^n\frac{q^{n^2}}{\bigl(-q; q^2\bigr)_n}
=\frac{\bigl(q, -q^2; q^2\bigr)_\infty }{\bigl(q^2, -q; q^2\bigr)_\infty}
\sum_{n=0}^\infty \sum_{j=0}^{2n} \bigl(1-q^{4n+2}\bigr) q^{3n^2+n-j(j+1)/2}
\end{align*}
and
\begin{equation}\label{Gauss:eqn8}
\sum_{n=0}^\infty (-1)^n q^{n^2}
=\frac{\bigl(q; q^2\bigr)_\infty}{\bigl(q^2; q^2\bigr)_\infty}
\sum_{n=0}^\infty \sum_{j=0}^{2n} \bigl(1-q^{4n+2}\bigr) q^{4n^2+2n-j(j+1)/2}.
\end{equation}

Replacing $q$ by $-q$ in (\ref{Gauss:eqn8}), we find that
\begin{equation*}%\label{Gauss:eqn9}
	\sum_{n=0}^\infty q^{n^2}
	=\frac{\bigl(-q; q^2\bigr)_\infty}{\bigl(q^2; q^2\bigr)_\infty}
	\sum_{n=0}^\infty \sum_{j=0}^{2n} (-1)^{j(j+1)/2}\bigl(1-q^{4n+2}\bigr) q^{4n^2+2n-j(j+1)/2},
\end{equation*}
which is different from the following identity \cite[Proposition~6.11]{Liu2013RamanJ}:
\begin{equation*}%\label{Gauss:eqn10}
\sum_{n=0}^\infty q^{n^2}
=\frac{1}{(q; q)_\infty}
\sum_{n=0}^\infty \sum_{j=-n}^{n} (-1)^{j}\bigl(1-q^{4n+2}\bigr) q^{3n^2+n-j(3j+1)/2}.
\end{equation*}

Clearly, we have not exhausted the application of Theorems~\ref{liuthmc} and \ref{liuthmd} in obtaining Rogers--Hecke type series. For further applications,
we need more terminating ${}_3\phi_2$ series identities.

Theorem~\ref{liuthm} clearly contains infinitely many $q$-formulas since we can choose $f$ in many different ways. For example, if we take
\[
f(x)=\prod_{j=1}^m \frac{(b_j x/q; q)_\infty} {(c_j x/q; q)_\infty}
\]
in Theorem~\ref{liuthm}, then we are led to the following theorem.
\begin{Theorem}%\label{mprodthm}
For $\{|\alpha a c_1/q|, |\alpha b_1|, \ldots, |\alpha a c_m/q|, |\alpha b_m|\}<1$, we have
\begin{gather*}
\prod_{j=1}^m \frac{(\alpha a b_j/q, \alpha c_j; q)_\infty}
{(\alpha a c_j/q, \alpha b_j; q)_\infty}
\\
\qquad{}=\sum_{n=0}^\infty \frac{(1-\alpha q^{2n})(\alpha, q/a; q)_n (a/q)^n}{(1-\alpha)(q, \alpha,\alpha a; q)_n}
{_{m+2}\phi_{m+1}}\biggl({{q^{-n}, \alpha q^n, \alpha c_1, \ldots, \alpha c_m}\atop{q\alpha, \alpha b_1, \ldots, \alpha b_m}}; q, q\biggr).
\end{gather*}
\end{Theorem}
We believe that the application of Theorem~\ref{liuthm} in $q$-series and number theory is still worth exploring.

\subsection*{Acknowledgments}
I am grateful to the referees and the editors for proposing many very helpful comments and suggestions. I also thank Dandan Chen, Chun Wang and Chang Xu for pointing out several misprints of an earlier version of this paper. This research is supported in part by the National Natural Science Foundation of China (Grant No.~12371328) and Science and Technology Commission of Shanghai Municipality (No.~22DZ2229014).

\pdfbookmark[1]{References}{ref}
\LastPageEnding

\end{document}